\documentclass[11pt,a4paper]{article}
\usepackage[T1]{fontenc}
\usepackage{url}
\usepackage{amssymb}
\usepackage{amsmath}
\usepackage{mathrsfs}
\usepackage[english]{babel}
\oddsidemargin=\evensidemargin
 
\newtheorem{theorem}{Theorem}[section]
\newtheorem{proposition}[theorem]{Proposition}
\newtheorem{lemma}[theorem]{Lemma}
\newtheorem{corollary}[theorem]{Corollary}
\newtheorem{definition}[theorem]{Definition}
\newtheorem{remark}[theorem]{Remark}
\newtheorem{example}[theorem]{Example}

\numberwithin{equation}{section}

\newcommand{\Rot}{\operatorname{\mathbf{curl}}}
\newcommand{\Div}{\operatorname{\mathrm{div}}}

\newcommand{\rot}{\operatorname{\mathrm{curl}}}

\newcommand{\loc}{\mathrm{loc}}
\renewcommand{\Im}{\operatorname{Im}}

\newcommand  {\C}{{\mathbb C}}
\newcommand  {\N}{{\mathbb N}}

\newcommand  {\R}{{\mathbb R}}

\newcommand {\Id}{\mathrm {I}}

\newcommand  {\D}{\mathsf D}

\renewcommand{\SS}{\boldsymbol{\mathsf S}}
\newcommand  {\TT}{\boldsymbol{\mathsf T}}

\newcommand  {\LL}{\boldsymbol{\mathsf L}}
\newcommand  {\HH}{\boldsymbol{\mathsf H}}
\newcommand  {\EE}{\boldsymbol{\mathsf E}}

\newcommand  {\Pp}{\boldsymbol{\mathsf P}} 

\newcommand {\Tr}{\operatorname{Trace}} 
\newcommand {\Jac}{\operatorname{Jac}} 
\newcommand {\com}{\operatorname{com}}

\newcommand  {\nn}{\boldsymbol{\mathsf n}}

\newcommand  {\uu}{\boldsymbol{\mathsf u}}
\newcommand  {\jj}{\boldsymbol{\mathsf j}}
\newcommand  {\mm}{\boldsymbol{\mathsf m}}
\newcommand  {\vv}{\boldsymbol{\mathsf v}}

\newenvironment{proof}{\begin{trivlist}
                       \item[]{\sc Proof. }}{\hfill $\blacksquare$
                     \end{trivlist}} 
  \newcommand{\transposee}[1]{{\vphantom{#1}}^{\mathit T}{#1}}    
\begin{document}

\title{Shape derivatives of boundary integral operators in electromagnetic scattering}
\author{M. Costabel $\sp 1$, F. Le Lou\"er $\sp 2$\\\small $\sp 1$ IRMAR, University of Rennes 1\\\small$\sp 2$ POEMS, INRIA-ENSTA,  Paris }
\date{}
\maketitle

\begin{abstract} We develop the shape  derivative analysis of solutions to the problem of scattering of time-harmonic electromagnetic waves by a bounded penetrable obstacle. Since boundary integral equations are a classical tool to solve  electromagnetic scattering problems, we study the shape differentiability properties of the standard electromagnetic boundary integral operators. To this end, we start with the G\^ateaux differentiability analysis  with respect to deformations of the obstacle of boundary integral operators with pseudo-homogeneous kernels acting between Sobolev spaces. The boundary integral operators of electromagnetism  are typically bounded on the space of tangential vector fields of mixed regularity  $\TT\HH\sp{-\frac{1}{2}}(\Div_{\Gamma},\Gamma)$. Using Helmholtz decomposition, we can base their analysis on the study of scalar integral operators in standard Sobolev spaces, but we then have to study the G\^ateaux differentiability of surface differential operators. We prove that the electromagnetic boundary integral operators are infinitely differentiable without loss of regularity and  that the solutions of the scattering problem are infinitely shape differentiable away from the boundary of the obstacle, whereas their derivatives lose regularity on the boundary. We  also give a characterization of the first  shape derivative  as a solution of a new electromagnetic scattering problem. 
\end{abstract}      


\tableofcontents

\newpage
\section*{Introduction}
Consider the scattering of time-harmonic electromagnetic waves by a bounded obstacle $\Omega$ in $\R^3$ with a smooth and simply connected boundary $\Gamma$ filled with an homogeneous dielectric material. This problem  is described by the system of Maxwell's equations, valid in the sense of distributions, with two transmission conditions on the boundary  or the obstacle guaranteeing the continuity of the tangential components of the electric and magnetic fields across the interface.  The transmission problem is completed by the Silver-M\"uller radiation condition at infinity (see \cite{Monk} and \cite{Nedelec}).  Boundary integral equations are an efficient  method to solve such problems for low and high frequencies. The dielectric scattering  problem is usually reduced to a system of two boundary integral equations for two unknown tangential vector fields on the interface (see \cite{BuffaHiptmairPetersdorffSchwab} and \cite{Nedelec}). We refer to  \cite{CostabelLeLouer} and \cite{CostabelLeLouer2} for methods developed by the authors to solve this problem using a single boundary integral equation.

 Optimal shape design with the modulus of the far field pattern of the  dielectric scattering problem as goal is of practical interest in some important fields of applied mathematics, as for example telecommunication systems and radars. The utilization of shape optimization methods  requires the analysis of the dependency of the solution on the shape of the dielectric scatterer.  An explicit form of the shape derivatives is required in view of their implementation in a shape optimization algorithms such as gradient methods or Newton's method.

 In this paper, we present a complete analysis of the shape differentiability of the solution of the dielectric scattering problem using an integral representation. Even if numerous works exist on the calculus of shape variations \cite{Hadamard, PierreHenrot, Pironneau, Simon, Zolesio}, in the framework of boundary integral equations the scientific literature is not extensive. However, one can cite the papers  \cite{Potthast2}, \cite{Potthast1} and  \cite{Potthast3}, where  R. Potthast has considered the question, starting with his PhD thesis \cite{Potthast4}, for the Helmholtz equation with Dirichlet or Neumann boundary conditions  and the perfect conductor problem, in spaces of continuous and H\"older continuous functions. Using the  integral representation of the solution, one is lead to study  the G\^ateaux differentiability of  boundary integral operators and potential operators with weakly and strongly singular kernels.
 
  The natural space of distributions (energy space) which occurs in the electromagnetic potential theory is $\TT\HH\sp{-\frac{1}{2}}(\Div_{\Gamma},\Gamma)$, the set of tangential vector fields whose components are in the Sobolev space $H\sp{-\frac{1}{2}}(\Gamma)$ and whose surface divergence is in $H\sp{-\frac{1}{2}}(\Gamma)$. 
We face two main difficulties: 
On one hand,  to be able to construct shape derivatives of the solution -- which is given in terms of products of boundary integral operators and their inverses -- it is imperative to prove that the derivatives are bounded operators between the same spaces as the boundary integral operators themselves. 
On the other hand, the very definition of shape differentiability of operators defined on $\TT\HH\sp{-\frac{1}{2}}(\Div_{\Gamma},\Gamma)$ poses non-trivial problems. Our approach consists in using the Helmholtz decomposition of this Hilbert space. In this way, we split the analysis in two main steps:  First the G\^ateaux differentiability analysis of scalar boundary integral operators and potential operators with pseudo-homogeneous kernels,  and second the study of derivatives with respect to smooth deformations of the obstacle of surface differential operators in the classical Sobolev spaces.
  
 This work contains results from the thesis \cite{FLL} where this analysis has been used to develop a shape optimization algorithm of dielectric lenses in order to obtain  a prescribed radiation pattern.
 \medskip
 
The paper is organized as follows: 

In section \ref{ScatProb} we define the scattering problem of time-harmonic electromagnetic waves at a dielectric interface and the appropriate spaces.  In section \ref{BoundIntOp}  we  recall some results about trace mappings and  boundary integral operators in electromagnetism, following the notation of \cite{BuffaHiptmairPetersdorffSchwab, Nedelec}. We then give an integral representation of the solution following \cite{CostabelLeLouer}.  In section \ref{ShapeD}, we introduce  the notion of shape derivative and its connection to G\^ateaux  derivatives. We also recall  elementary results about differentiability in Fr\'echet spaces. 

The section \ref{GDiffPH} is dedicated to the G\^ateaux differentiability analysis of a class of boundary integral operators with respect to deformations of the boundary. We generalize  the results proved in \cite{Potthast2, Potthast1} for the standard acoustic boundary integral operators,  to  the class of integral operators with pseudo-homogenous kernels. We also give higher order G\^ateaux derivatives  of coefficient functions such as the Jacobian of the change of variables associated with the deformation,  or the components of the unit normal vector. These results are new and allow us to obtain explicit forms of the derivatives of the integral operators. 

The last section contains the main results of this paper: the shape differentiability properties of the solution of the dielectric scattering problem.  We begin by discussing the difficulties of defining the shape dependency of operators defined on the shape-dependent space $\TT\HH\sp{-\frac{1}{2}}(\Div_{\Gamma},\Gamma)$, and we present an altervative  using the Helmholtz decomposition  (see \cite{delaBourdonnaye}) on the boundary of smooth domains. We then analyze the differentiability of a family of surface differential operators. Again we prove their infinite G\^ateaux differentiability and give an explicit expression of their derivatives. These results are new and important for the numerical implementation of the shape derivatives. Using the chain rule, we deduce  the infinite shape differentiability of the solution of the scattering problem away from the boundary and  an expression of the shape derivatives.  More precisely, we prove that  the boundary integral operators are infinitely G\^ateaux differentiable without loss of regularity, whereas previous results allowed such a loss \cite{Potthast3}, and we prove that the shape derivatives of the potentials are smooth far from the boundary but they lose regularity in the neighborhood of the boundary.   

These new results generalize existing results: In the acoustic case, using the variational formulation, a characterization of the first G\^ateaux derivative was given by A. Kirsch in \cite{Kirsch} for the Dirichlet problem and then for a transmission problem by F. Hettlich in \cite{Hettlich, HettlichErra}. R. Potthast used the integral equation method to obtain a characterisation of the first shape derivative of the solution of the perfect conductor scattering problem. 

We end the paper by  formulating a characterization of the first shape derivative as the solution of a new electromagnetic scattering problem. We show that both by directly deriving the boundary values and by using the integral representation of the solution, we obtain the same characterization.

%%%%%%%%%%%%%%%%%%%%%%%%%%%%%%%%%
\section{The dielectric scattering problem} \label{ScatProb}
Let $\Omega$ denote a bounded domain in $\R^{3}$  and let $\Omega^c$ denote the exterior domain $\R^3\backslash\overline{\Omega}$. 
In this paper, we will assume that the boundary $\Gamma$ of $\Omega$ is a smooth and simply connected closed surface, so that $\Omega$ is diffeomorphic to a ball. Let $\nn$ denote the outer unit normal vector on the boundary $\Gamma$.

In $\Omega$ (resp. $\Omega^c$) the electric permittivity $\epsilon_{i}$ (resp.
$\epsilon_{e}$) and the magnetic permeability $\mu_{i}$ (resp. $\mu_{e}$) are positive constants. The frequency $\omega$ is the same in $\Omega$ and in
$\Omega^c$.  The interior wave number $\kappa_{i}$ and the exterior wave   number $\kappa_{e}$ are complex constants of non negative imaginary part.

\textbf{Notation: }
For a domain $G\subset\R^3$ we denote by $H^s(G)$ the usual $L^2$-based Sobolev space of order $s\in\R$, and by $H^s_{\loc}(\overline G)$ the space of functions whose restrictions to any bounded subdomain $B$ of $G$ belong to $H^s(B)$. Spaces of vector functions will be denoted by boldface letters, thus 
$$
 \HH^s(G)=(H^s(G))^3\,.
$$ 
If $\D$ is a differential operator, we write:
\begin{eqnarray*}
\HH^s(\D,\Omega)& = &\{ u \in \HH^s(\Omega) : \D u \in \HH^s(\Omega)\}\\
\HH^s_{\loc}(\D,\overline{\Omega^c})&=& \{ u \in \HH^s_{\loc}(\overline{\Omega^c}) : \D u \in
\HH^s_{\loc}(\overline{\Omega^c}) \}\end{eqnarray*}
The space $\HH^s(\D, \Omega)$ is endowed with the natural graph norm. When $s=0$, this defines in particular the Hilbert spaces $\HH(\Rot,\Omega)$ and $\HH(\Rot\Rot,\Omega)$.
 We denote the $\LL^2$ scalar product on $\Gamma$ by $\langle\cdot,\cdot\rangle_{\Gamma}$.

\bigskip

The time-harmonic Maxwell's sytem can be reduced to second order equations for the electric field only.  The time-harmonic dielectric scattering problem is then formulated as follows.

\medskip
 
\textbf{The dielectric scattering problem :} 
Given an incident field 
$\EE^{inc}\in\HH_{\loc}(\Rot,\R^3)$ that satisfies 
$\Rot\Rot \EE^{inc} - \kappa_{e}^2\EE^{inc} =0$
in a neighborhood of $\overline{\Omega}$,
we seek  two fields $\EE^{i}\in \HH(\Rot,\Omega)$ and $\EE^{s}\in
\HH_{\loc}(\Rot,\overline{\Omega^c})$ satisfying the time-harmonic Maxwell equations
\begin{eqnarray}
\label{(1.2a)}
  \Rot\Rot \EE^{i} - \kappa_{i}^2\EE^{i}& = 0&\text{ in }\Omega,
\\
\label{(1.2b)} 
  \Rot\Rot \EE^{s} - \kappa_{e}^2\EE^{s} &= 0&\text{ in }\Omega^c,\end{eqnarray}
the two transmission conditions, 
\begin{eqnarray}{}
\label{T1}&\,\,\;\;\nn\times \EE^{i}=\nn\times( \EE^{s}+\EE^{inc})&\qquad\text{ on }\Gamma
\\
\label{T2} &\mu_{i}^{-1}(\nn\times\Rot \EE^{i}) = \mu_{e}^{-1}\nn\times\Rot(\EE^{s}+\EE^{inc})&\qquad\text{ on }\Gamma\end{eqnarray} 
and the Silver-M\"uller radiation condition:
\begin{equation}\label{T3}\lim_{|x|\rightarrow+\infty}|x|\left| \Rot \EE^{s}(x)\times\frac{x}{| x |}- i\kappa_{e}\EE^{s}(x) \right|
 =0.\end{equation}
The interior and exterior magnetic fields are then given by $\HH^i=\dfrac{1}{i\omega\mu_{i}}\EE^{i}$ and $\HH^s=\dfrac{1}{i\omega\mu_{e}}\EE^{s}$.
It is well known that this problem admit a unique solution for any positive real values of the exterior wave number \cite{BuffaHiptmairPetersdorffSchwab, FLL, Nedelec}.

An important quantity, which is of interest in many shape optimization problems, is the far field pattern of the electric solution, defined on the unit sphere of $\R^3$, by

$$\EE^{\infty}(\hat{x})=\lim_{|x|\rightarrow\infty}4\pi|x|\dfrac{\EE^s(x)}{e^{i\kappa_{e}|x|}},\qquad \text{ with }\dfrac{x}{|x|}=\hat{x}.$$
%%%%%%%%%%%%%%%%%%%%%%%%%%%%%%%%%%
\section{Boundary integral operators and main properties} \label{BoundIntOp}

%+++++++++++++++++++++++++++++++++++++++++
\subsection{Traces and tangential differential calculus}
We use surface differential operators and traces.  More details can be found in \cite{MartC,Nedelec}.

For a vector function $\vv\in(\mathscr{C}^k(\R^3))^q$ with $k,q\in\N^*$, we note $[\nabla\vv]$ the matrix whose the $i$-th column is the gradient of the $i$-th component of $\vv$ and we set $[\D\vv]=\transposee{[\nabla\vv]}$.
The tangential gradient of any scalar function $u\in\mathscr{C}^k(\Gamma)$ is defined by 
\begin{equation}\label{G}\nabla_{\Gamma}u=\nabla\tilde{u}_{|\Gamma}-\left(\nabla\tilde{u}_{|\Gamma}\cdot\nn\right)\nn,\end{equation}
and the tangential vector curl  by 
\begin{equation}\label{RR}\Rot_{\Gamma}u=\nabla\tilde{u}_{|\Gamma}\times\nn,\end{equation}
where $\tilde{u}$ is an extension of $u$ to the whole space $\R^3$.
For a vector function $\uu\in(\mathscr{C}^k(\Gamma))^3$, we note $[\nabla_{\Gamma}\uu]$ the matrix whose the $i$-th column is the tangential gradient of the $i$-th component of $\uu$ and we set  $[\D_{\Gamma}\uu]=\transposee{[\nabla_{\Gamma}\uu]}$.

We define the surface divergence of any vectorial function $\uu\in(\mathscr{C}^k(\Gamma))^3$   by \begin{equation}\label{D}\Div_{\Gamma}\uu=\Div\tilde{\uu}_{|\Gamma}-\left([\nabla\tilde{\uu}_{|\Gamma}]\nn\cdot\nn\right),\end{equation}
 and the surface scalar curl $\rot_{\Gamma_{r}}\uu_{r}$ by 
\begin{equation*}\label{R}\rot_{\Gamma}\uu=\nn\cdot\left(\Rot\tilde{\uu})\right)\end{equation*}
where $\tilde{\uu}$ is an extension of $\uu$ to the whole space $\R^3$.
These definitions do not depend on the extension. 

\begin{definition}\label{2.1} For  a vector function $\vv\in (\mathscr{C}^{\infty}(\overline{\Omega}))^3$ and a scalar function $v\in\mathscr{C}^{\infty}(\overline{\Omega})$ we define the
traces :\newline
$$\gamma v=v_{|_{\Gamma}} ,$$
$$\gamma_{D}\vv:=(\nn\times \vv)_{|_{\Gamma}}\textrm{ (Dirichlet) and}$$
$$\gamma_{N_{\kappa}}\vv:=\kappa^{-1}(\nn\times\Rot \vv)_{|_{\Gamma}}\textrm{ (Neumann).}$$
\end{definition}

We introduce the Hilbert spaces $H^s(\Gamma)=\gamma\left(H^{s+\frac{1}{2}}(\Omega)\right),$ and 
$\TT\HH^{s}(\Gamma)
=\gamma_{D}\left(\HH^{s+\frac{1}{2}}(\Omega)\right)$
For $s>0$, 
 the traces 
 $$\gamma:H^{s+\frac{1}{2}}(\Omega)\rightarrow H^{s}(\Gamma),$$ 
 $$\gamma_{D}:\HH^{s+\frac{1}{2}}(\Omega)\rightarrow \TT\HH^{s}(\Gamma)
 $$ 
 are then continuous. 
The dual of $H^s(\Gamma)$ and $\TT\HH^{s}(\Gamma)$ 
   with respect to the $L^2$ (or $\LL^2$) scalar product is denoted by $H^{-s}(\Gamma)$ and $\TT\HH^{-s}(\Gamma)$, respectively. 
 
The surface differential operators defined here above can be extended to the Sobolev spaces: The tangential gradient and the tangential vector curl are linear and continuous from $H^{s+1}(\Gamma)$ to $\TT\HH^s(\Gamma)$,  the surface divergence and  the  surface scalar curl are linear and continuous from $\TT\HH^{s+1}(\Gamma)$ to $H^s(\Gamma)$.

\begin{definition} We define the Hilbert space
$$
  \TT\HH^{-\frac{1}{2}}(\Div_{\Gamma},\Gamma)=\left\{ \jj\in
\TT\HH^{-\frac{1}{2}}(\Gamma),\Div_{\Gamma}\jj \in
H^{-\frac{1}{2}}(\Gamma)\right\}
$$
endowed with the norm 
$$
||\cdot||_{\TT\HH^{-\frac{1}{2}}(\Div_{\Gamma},\Gamma)}=||\cdot||_{\TT\HH^{-\frac{1}{2}}(\Gamma)}+||\Div_{\Gamma}\cdot||_{H^{-\frac{1}{2}}(\Gamma)}.
$$ 
\end{definition}

\begin{lemma}
\label{2.3} The operators $\gamma_{D}$ and $\gamma_{N}$ are linear and
continuous from $\mathscr{C}^{\infty}(\overline{\Omega},\R^3)$ to $\TT\LL^2(\Gamma)$ and they
can be extended to continuous linear operators from $\HH(\Rot,\Omega)$ and 
$\HH(\Rot,\Omega)\cap\HH(\Rot\Rot,\Omega)$, respectively, to
$\TT\HH^{-\frac{1}{2}}(\Div_{\Gamma},\Gamma)$. 
\end{lemma}

For  $\uu\in \HH_{\loc}(\Rot,\overline{\Omega^c})$ and $\vv \in
\HH_{\loc}(\Rot\Rot,\overline{\Omega^c}))$
we define $\gamma_{D}^c\uu $ and $\gamma_{N}^c\vv $ in the same way and the same mapping properties hold true.

Recall that we assume that the boundary $\Gamma$ is smooth and topologically trivial. For a proof of the following result, we refer to \cite{BuffaCiarlet,MartC,Nedelec}.
\begin{lemma}
\label{LapBel}
 Let $t\in\R$. The Laplace-Beltrami operator \begin{equation}\label{eqd1}\Delta_{\Gamma}=\Div_{\Gamma}\nabla_{\Gamma}=-\rot_{\Gamma}\Rot_{\Gamma}\end{equation} is linear and continuous from $H^{t+2}(\Gamma)$ to $H^t(\Gamma)$. \\
It is an isomorphism from $H^{t+2}(\Gamma)\slash\R$ to the space
$H^t_{*}(\Gamma)$ defined by
$$
u\in H^t_{*}(\Gamma)\quad\Longleftrightarrow\quad u\in H^t(\Gamma)\textrm{ and }\int_{\Gamma}u=0.
$$
\end{lemma}
This result is due to the surjectivity of the operators $\Div_{\Gamma}$ and $\rot_{\Gamma}$ from $\TT\HH^{t+1}(\Gamma)$ to $H^t_{*}(\Gamma)$. 
 
We note the following equalities:
\begin{equation}
\label{eqd2}
  \rot_{\Gamma}\nabla_{\Gamma}=0\text{ and }\Div_{\Gamma}\Rot_{\Gamma}=0
\end{equation}
\begin{equation}
\label{eqd3}
 \Div_{\Gamma}(\nn\times \jj)=-\rot_{\Gamma}\jj\text{ and }\rot_{\Gamma}(\nn\times \jj)=\Div_{\Gamma}\jj\end{equation}

%++++++++++++++++++++++++++++++++
\subsection{Pseudo-homogeneous kernels}
 In this paper we are concerned with boundary integral operators of the form :
\begin{equation}\label{p}\mathcal{K}_{\Gamma}u(x)=\mathrm{vp.}\int_{\Gamma}k(y,x-y)u(y)d\sigma(y),\;x\in\Gamma\end{equation} 
where the integral is assumed to exist in the sense of  a Cauchy principal value and the kernel  $k$ is weakly singular, regular with respect to the variable  $y\in\Gamma$ and quasi-homogeneous with respect to the variable  $z=x-y\in\R^3$. We recall the regularity properties of these operators on the Sobolev  spaces  $H^s(\Gamma)$, $s\in\R$ available also for their adjoints operators:
\begin{equation}\mathcal{K}_{\Gamma}^*(u)(x)=\mathrm{vp.}\int_{\Gamma}k(x,y-x)u(y)d\sigma(y),\;x\in\Gamma.\end{equation}
We use the class of weakly singular kernel introduced by Nedelec  (\cite{Nedelec} p. 176). 
 More details can be found in \cite{Eskin, Hormander, Journe, Meyer, Taylor1, Taylor2}.
\begin{definition} The homogeneous kernel $k(y,z)$ defined on $\Gamma\times\left(\R^3\backslash\{0\}\right)$ is said of class $-m$  with $m\geq0$ if
$$\left\{\begin{array}{c}\sup\limits_{y\in\R^d}\sup\limits_{|z|=1}\left|\dfrac{\partial^{|\alpha|}}{\partial y^{\alpha}}\dfrac{\partial^{|\beta|}}{\partial z^{\beta}}k(y,z)\right|\leq C_{\alpha,\beta}, \;\text{for all multi-index }\alpha\text{ and }\beta,\\\\
\dfrac{\partial^{|\beta|}}{\partial z^{\beta}}k(y,z)\text{ is homogeneous of degree }-2\text{ with respect to the  variable }z\\\\\text{ for all }|\beta|=m 
\text{ and }D_{z}^mk(y,z)\text{ is odd with respect to the  variable }z.\end{array}\right.$$
\end{definition}
\begin{definition}  The kernel $k\in\mathscr{C}^{\infty}\left(\Gamma\times\left(\R^3\backslash\{0\}\right)\right)$ is pseudo-homogeneous of class $-m$ for an integer  $m$ such that  $m\geqslant0$, if for all integer $s$ the kernel $k$ admit the following asymptotic expansion when $z$ tends to $0$:
\begin{equation}\label{(dev)}k(y,z)=k_{m}(y,z)+\sum_{j=1}^{N-1}k_{m+j}(y,z)+k_{m+N}(y,z),\end{equation}
where for $j=0,1,...,N-1$ the function $k_{m+j}$ is homogeneous of class  $-(m+j)$  and $N$ is chosen such that $k_{m+N}$ is $s$ times differentiables.\end{definition}

 For the proof of the following theorem, we refer to \cite{Nedelec}. 
\begin{theorem}Let $k$ be a  pseudo-homogeneous kernel of class $-m$. The  associated operator $\mathcal{K}_{\Gamma}$ given by \eqref{p} is linear and continuous from $H^{s}(\Gamma)$ to $H^{s+m}(\Gamma)$ for all $s\in\R$.\end{theorem}
We have similar results for the adjoint operators $\mathcal{K}_{\Gamma}^{*}$.\newline

 The following theorem is established in \cite{Eskin}.
\begin{theorem}Let  $k$ be a pseudo-homogeneous kernel of class $-m$. The potential operator  $\mathcal{P}$ defined by 
\begin{equation}\label{p'}\mathcal{P}(u)(x)=\int_{\Gamma}k(y,x-y)u(y)d\sigma(y),\;x\in\R^3\backslash\Gamma\end{equation}is continuous from $H^{s-\frac{1}{2}}(\Gamma)$ to $H^{s+m}(\Omega)\cup H^{s+m}_{loc}(\Omega^c)$ for all positive real number $s$.\end{theorem}

%++++++++++++++++++++++++++++++
\subsection{The electromagnetic boundary integral operators}
We use some well known results about electromagnetic potentials. Details can be found 
in \cite{BuffaCiarlet, BuffaCostabelSchwab, BuffaCostabelSheen,BuffaHiptmairPetersdorffSchwab, Nedelec}.

Let $\kappa$ be a complex number such that $\Im(\kappa)\ge0$ and let 
$$
  G(\kappa,|x-y|)=\dfrac{e^{i\kappa|x-y|}}{4\pi| x-y|}
$$
be the fundamental solution of the Helmholtz equation 
$$
  {\Delta u + \kappa^2u =0}.
$$
The single layer potential  $\psi_{\kappa}$ is given by   :
\begin{center}$(\psi_{\kappa}u)(x) = \displaystyle{\int_{\Gamma}G(\kappa,|x-y|)u(y) d\sigma(y)}\qquad x
\in\R^3\backslash\Gamma$,\end{center}
and its trace by 
$$
 V_{\kappa}u(x)= \int_{\Gamma}G(\kappa,|x-y|)u(y) d\sigma(y)\qquad x
\in\Gamma.
$$
The fundamental solution is pseudo-homogeneous of class $-1$ (see \cite{HsiaoWendland,Nedelec}). As consequence we have the following result :
\begin{lemma}
\label{3.1}  Let $s\in\R$.
The operators
$$ 
\begin{array}{ll} \psi_{\kappa} & : H^{s-\frac{1}{2}}(\Gamma)\rightarrow H^{s+1}_{\loc}(\R^3) \vspace{2mm} \\ 
V_{\kappa}& : H^{s-\frac{1}{2}}(\Gamma)\rightarrow H^{s+\frac{1}{2}}(\Gamma) \end{array}
$$  
are continuous.
\end{lemma}
We define the electric potential $\Psi_{E_{\kappa}}$ generated by $\jj\in
\TT\HH^{-\frac{1}{2}}(\Div_{\Gamma},\Gamma) $ by 
$$\Psi_{E_{\kappa}}\jj :=
\kappa\psi_{\kappa}\jj + \kappa^{-1}\nabla\psi_{\kappa}\Div_{\Gamma}\jj
$$
This can be written as $\Psi_{E_{\kappa}}\jj :=
\kappa^{-1}\Rot\Rot\psi_{\kappa}\jj$ because of the Helmholtz equation and
the identity $\Rot\Rot = -\Delta +\nabla\Div$ (cf \cite{BuffaCiarlet}).

We define the magnetic potential  
$\Psi_{M_{\kappa}}$ generated by $\mm\in \TT\HH^{-\frac{1}{2}}(\Div_{\Gamma},\Gamma) $
by 
$$
 \Psi_{M_{\kappa}}\mm := \Rot\psi_{\kappa}\mm.
$$
 We denote the identity operator by $\Id$.
\begin{lemma}
\label{3.2}
The potentials $\Psi_{E_{\kappa}}$ et $\Psi_{M_{\kappa}}$ are
continuous from $\TT\HH^{-\frac{1}{2}}(\Div_{\Gamma},\Gamma)$ to
$\HH_{\loc}(\Rot,\R^3)$. 
For $\jj\in \TT\HH^{-\frac{1}{2}}(\Div_{\Gamma},\Gamma)$ we have 
$$(\Rot\Rot -\kappa^2\Id)\Psi_{E_{\kappa}}\jj = 0\textrm{ and }(\Rot\Rot
-\kappa^2\Id)\Psi_{M_{\kappa}}\mm = 0\textrm{ in }\R^3\backslash\Gamma$$ and  $\Psi_{E_{\kappa}}\jj$
and $\Psi_{M_{\kappa}}\mm$ satisfy the Silver-M\"uller condition.
\end{lemma}
We define the electric and the magnetic far field operators for $\jj\in \TT\HH^{-\frac{1}{2}}(\Div_{\Gamma},\Gamma)$ and an element $\hat{x}$ of the unit sphere $S^2$ of $\R^3$ by

\begin{equation}\label{FF}\begin{split}\Psi_{E_{\kappa}}^{\infty}\,\jj(\hat{x})=&\;\kappa\;\hat{x}\times\left(\int_{\Gamma}e^{-i\kappa\hat{x}\cdot y}\jj(y)d\sigma(y)\right)\times\hat{x},\\\Psi_{M_{\kappa_{e}}}^{\infty}\,\jj(\hat{x})=&\;i\kappa\;\hat{x}\times\left(\int_{\Gamma}e^{-i\kappa\hat{x}\cdot y}\jj(y)d\sigma(y)\right).\end{split}\end{equation}
 These operators are bounded from $\TT\HH^{-\frac{1}{2}}(\Div_{\Gamma},\Gamma)$ to $\TT(\mathscr{C}^{\infty}(S^2))^3$.

We can now define the main boundary integral operators:
\begin{center}$C_{\kappa} = -\frac{1}{2}\{\gamma_{D}+\gamma_{D}^c\}\Psi_{E_{\kappa}} =
-\frac{1}{2}\{\gamma_{N}+\gamma_{N}^c\}\Psi_{M_{\kappa}}$,\end{center}
\begin{center}$M_{\kappa} = -\frac{1}{2}\{\gamma_{D}+\gamma_{D}^c\}\Psi_{M_{\kappa}} =
-\frac{1}{2}\{\gamma_{N}+\gamma_{N}^c\}\Psi_{E_{\kappa}}$.
\end{center}
These are bounded operators in 
$\TT\HH^{-\frac{1}{2}}(\Div_{\Gamma},\Gamma)$.
We have
 $$\begin{array}{ll}\hspace{-2mm}\operatorname{C_{\kappa}}\jj(x)\hspace{-2mm}&=-\displaystyle{\kappa\hspace{-1mm}\int_{\Gamma}\nn(x)\times(G(\kappa,|x-y|)\jj(y))d\sigma(y)+\kappa^{-1}\hspace{-2mm}\int_{\Gamma}\Rot_{\Gamma}^x(G(\kappa,|x-y|)\Div_{\Gamma}\jj(y))d\sigma(y)}\vspace{1mm}\\&=\left(-\kappa\;\;\nn\times V_{\kappa}\,\jj+\kappa^{-1}\rot_{\Gamma}V_{\kappa}\Div_{\Gamma}\jj\right)(x)\end{array}$$ 
 and
  $$\begin{array}{ll}\operatorname{M_{\kappa}}\jj(x)&=-\displaystyle{\int_{\Gamma}\nn(x)\times\Rot^x(G(\kappa,|x-y|)\jj(y))d\sigma(y)}\vspace{1mm}\\&=\;\;(\operatorname{D_{\kappa}}\jj-\operatorname{B_{\kappa}}\jj)(x),\end{array}$$ 
  with
   $$\begin{array}{ll}\operatorname{B_{\kappa}}\jj(x)&=\displaystyle{\int_{\Gamma}\nabla^xG(\kappa,|x-y|)\left(\jj(y)\cdot\nn(x)\right)d\sigma(y),}\vspace{2mm}\\\operatorname{D_{\kappa}}\jj(x)&=\displaystyle{\int_{\Gamma}\left(\nabla^xG(\kappa,|x-y|)\cdot\nn(x)\right)\jj(y)d\sigma(y).}\end{array}$$ 
The kernel of $D_{\kappa}$ is pseudo-homogeneous of class $-1$ and the operator $M_{\kappa}$ has the same regularity as $D_{\kappa}$ on $\TT\HH^{-\frac{1}{2}}(\Div_{\Gamma},\Gamma)$, that  is compact.
 \bigskip

 We describe briefly the boundary integral equation method developped by the autors \cite{CostabelLeLouer} to solve the dielectric scattering problem.
 
 \medskip
 
 \textbf{Boundary integral equation method :}
 This is based on  the Stratton-Chu formula, the jump relations of the electromagnetic potentials and the Calder\'on projector's formula (see \cite{BuffaHiptmairPetersdorffSchwab, Nedelec}).
 
 We  need  a variant of the operator $\operatorname{C_{\kappa}}$ defined for $\jj\in \TT\HH^{-\frac{1}{2}}(\Div_{\Gamma},\Gamma)$ by :
$$\operatorname{C_{0}^*}\jj=\;\nn\times V_{0}\,\jj+\Rot_{\Gamma}V_{0}\Div_{\Gamma}\jj.$$
The operator $C_{0}^*$ is bounded in $\TT\HH^{-\frac{1}{2}}(\Div_{\Gamma},\Gamma)$. 
We use the following ansatz on the integral representation of the exterior  electric field $\EE^s$:
\begin{equation} \EE^{s} =  -\operatorname{\Psi_{E_{\kappa_{e}}}}\jj -i\eta\operatorname{\Psi_{M_{\kappa_{e}}}}\operatorname{C_{0}^*}\jj\text{ in }\R^3\backslash\bar{\Omega} \end{equation}
$\eta$ is a positive real number and $\jj\in \TT\HH^{-\frac{1}{2}}(\Div_{\Gamma},\Gamma)$.
Thanks to the transmission conditions we have the integral representation of the interior field
\begin{equation} \label{rhd} \EE_{1} =-\frac{1}{\rho}(\Psi_{E_{\kappa_{i}}}\{\gamma_{N_{e}}^c\EE^{inc} +\operatorname{N_{e}}\jj\}) -
(\Psi_{M_{\kappa_{i}}}\{\gamma_{D}^c\EE^{inc} + \operatorname{L_{e}}\jj\})\text{ in }\Omega\end{equation}
 where $\rho=\dfrac{\kappa_{i}\mu_{e}}{\kappa_{e}\mu_{i}}$ and 
 $$ \operatorname{L_{e}}=\operatorname{C_{\kappa_{e}}}-i\eta\left(\frac{1}{2}\Id-\operatorname{M_{\kappa_{e}}}\right)\operatorname{C_{0}^*},$$ $$\operatorname{N_{e}}=\left(\frac{1}{2}\Id-\operatorname{M_{\kappa_{e}}}\right)+i\eta \operatorname{C_{\kappa_{e}}}\operatorname{C_{0}^*}.$$
We apply the exterior Dirichlet trace to the righthandside \eqref{rhd}. The density $\jj$ then solves the following boundary integral equation:

\begin{equation*} \operatorname{\SS} \jj = \rho\left(-\frac{1}{2}I + \operatorname{M_{\kappa_{i}}}\right)\operatorname{L_{e}}\jj + \operatorname{C_{\kappa_{i}}}\operatorname{N_{e}}\jj = -\rho\left(-\frac{1}{2}I + \operatorname{M_{\kappa_{i}}}\right)\gamma_{D}\EE^{inc}+ \operatorname{C_{\kappa_{i}}}\gamma_{N_{\kappa_{e}}}\EE^{inc}\text{ sur }\Gamma.\end{equation*} 
 The operator $\SS$ is linear, bounded and invertible on $\TT\HH^{-\frac{1}{2}}(\Div_{\Gamma},\Gamma)$.

If we are concerned with the far field pattern $\EE^{\infty}$ of the solution, it suffices to replace the potential operators $\Psi_{E_{\kappa_{e}}}$ and $\Psi_{M_{\kappa_{e}}}$ by the far field operators $\Psi_{E_{\kappa_{e}}}^{\infty}$ and $\Psi_{M_{\kappa_{e}}}^{\infty}$ respectively.

The solution $\EE(\Omega)=(\EE^{i}(\Omega),\EE^{s}(\Omega))$ and the far field pattern $\EE^{\infty}(\Omega)$ consists of applications defined by integrals on the boundary $\Gamma$ and if the incident field is a fixed data,   these quantities depend on the scatterrer $\Omega$ only.

%%%%%%%%%%%%%%
\section{Some remarks on shape derivatives} \label{ShapeD}
We want to study the dependance of any functionals $F$  with respect to the shape of the dielectric scatterer $\Omega$. The  $\Omega$-dependance is highly nonlinear. The standard differential calculus tools  need the framework of  topological vector  spaces which are locally convex at least \cite{Schwartz}, framework we do not dispose in the case of shape functionals. An interesting approach consists in representing the variations of the domain $\Omega$ by elements of a function space. We consider  variations generated by transformations of the form $$x\mapsto x+r(x)$$ of any points $x$ in the space $\R^3$, where $r$ is a vectorial function defined (at least) in the neiborhood of $\Omega$. This transformation deforms the domain $\Omega$  in a  domain $\Omega_{r}$ of boundary $\Gamma_{r}$. The functions $r$ are assumed to be a  small enough elements of a Fr\'echet space $\mathcal{X}$ in order that $(\Id+r)$ is an isomorphism from $\Gamma$ to  $$\Gamma_{r}=(\Id+r)\Gamma=\left\{x_{r}=x+r(x); x\in\Gamma\right\}.$$

 Since we consider smooth surfaces, in the remaining of this paper, the space $\mathcal{X}$ will be  the Fr\'echet  space 
$\mathscr{C}^{\infty}_{b}(\R^3,\R^3)=\bigcap\limits_{k\in\N} \mathscr{C}^k_{b}(\R^3,\R^3)$ undowed with the set of non decreasing seminorms $(||\cdot||_{k})_{k\in\N}$ where
 $\mathscr{C}^k_{b}(\R^3,\R^3)$ with $k\in\N$ is  the space of $k$-times continuously differentiable  functions whose the derivatives are bounded  and  $$ ||r||_{k}=\sup_{0\leq p\leq k}\;\sup_{x\in\R^3}\left|r^{(p)}(x)\right|.$$
For $\epsilon$ small enough we set $$B^{\infty}_{\epsilon}= \left\{r\in\left(\mathscr{C}^{\infty}(\overline{\R^3})\right)^3,\;d_{\infty}(0, r)<\epsilon\right\},$$
where $d_{\infty}$ is the metric induced by the seminorms.

 We introduce the application $$ r\in B^{\infty}_{\epsilon}\mapsto\mathcal{F}_{\Omega}(r)=F(\Omega_{r}).$$ We define the shape derivative of the  functional $F$ trough the deformation $\Omega\rightarrow\Omega_{\xi}$ as  the G\^ateaux derivative of the application $\mathcal{F}_{\Omega}$ in the direction $\xi\in\mathcal{X}$.
We write: $$DF[\Omega;\xi]=\frac{\partial}{\partial t}_{|t=0}\mathcal{F}_{\Omega}(t\xi).$$

%++++++++++++++++++++++++++++++++++++++++++
\subsection{G\^ateaux differentiability: elementary results}
   Fr\'echet spaces are locally convex, metrisable and complete  topological vector spaces  on which we can extend any elementary results available on Banach spaces. We recall some of them. We refer to the Schwarz's book \cite{Schwartz} for more details.
\medskip

Let $\mathcal{X}$ and $\mathcal{Y}$ be Fr\'echet spaces and let $U$ be a subset of $\mathcal{X}$.
  \begin{definition}(\textbf{G\^ateaux semi-derivatives}) 
 The application  $f\;:U\rightarrow \mathcal{Y}$ is said to have  G\^ateaux semiderivative at $r_{0}\in U$ in the direction $\xi\in\mathcal{X}$ if the following limit exists and is finite $$\frac{\partial}{\partial r}f[r_{0};\xi]=\lim_{t\rightarrow0}\dfrac{f(r_{0}+t\xi)-f(r_{0})}{t}=\frac{\partial}{\partial t}_{\big|t=0}f(r_{0}+t\xi).$$\end{definition}

  \begin{definition}(\textbf{G\^ateaux differentiability}) 
   The application  $f\;:U\rightarrow \mathcal{Y}$ is said to be  G\^ateaux differentiable at $r_{0}\in U$ if it has G\^ateaux semiderivatives in all direction $\xi\in\mathcal{X}$ and if the map $$\xi\in\mathcal{X}\mapsto\frac{\partial}{\partial r}f[r_{0};\xi]\in\mathcal{Y}$$ is linear and continuous.\end{definition}
 We say that $f$ is continuously (or $\mathscr{C}^1$-) G\^ateaux differentiable if it is  G\^ateaux differentiable at all $r_{0}\in U$ and the application $$\frac{\partial}{\partial r}f : (r_{0};\xi)\in U\times\mathcal{X}\mapsto\frac{\partial}{\partial r}f[r_{0};\xi]\in\mathcal{Y}$$is continuous.
   \begin{remark}Let us come  to shape functionals. In calculus of shape variation, we usually consider  the G\^ateaux derivative in $r=0$  only. This is due to the result : If $\mathcal{F}_{\Omega}$ is  G\^ateaux differentiable on $B^{\infty}_{\epsilon}$ then for all $\xi\in\mathcal{X}$ we have
   $$\frac{\partial}{\partial r} \mathcal{F}_{\Omega}[r_{0};\xi]=\D F(\Omega_{r_{0}};\xi\circ(\Id+r_{0})^{-1})=\frac{\partial}{\partial r} \mathcal{F}_{\Omega_{r_{0}}}[0;\xi\circ(\Id+r_{0})^{-1}].$$\end{remark}

\begin{definition}(\textbf{higher order derivatives}) 
Let $m\in\N$.  We say that $f$ is $(m+1)$-times continuously (or $\mathscr{C}^{m+1}$-) G\^ateaux differentiable  if it is $\mathscr{C}^{m}$-G\^ateaux differentiable  and $$r\in U\mapsto\dfrac{\partial^{m} }{\partial r^{m}}f[r;\xi_{1},\hdots,\xi_{m}]$$ is continuously G\^ateaux differentiable for all $m$-uple $(\xi_{1},\hdots,\xi_{m})\in\mathcal{X}^{m}$. Then for all $r_{0}\in U$ the  application
$$(\xi_{1},\hdots,\xi_{m+1})\in\mathcal{X}^{m+1}\mapsto\frac{\partial^{m+1}}{\partial r^{m+1}}f[r_{0};\xi_{1},\hdots,\xi_{m+1}]\in\mathcal{Y}$$ is $(m+1)$-linear, symetric and continuous.   We say that $f$ is $\mathscr{C}^\infty$-G\^ateaux differentiable if it is $\mathscr{C}^m$-G\^ateaux differentiable for all $m\in\N$.
 \end{definition}

 We use the notation 
\begin{equation}\frac{\partial^m}{\partial r^m}f[r_{0},\xi]=\frac{\partial^m}{\partial t^m}_{\big|t=0}f(r_{0}+t\xi).\end{equation}
 If it is $\mathscr{C}^m$-G\^ateaux differentiable we have
 \begin{equation}\label{Gsym}\frac{\partial^m}{\partial r^m}f[r_{0},\xi_{1},\hdots,\xi_{m}]=\frac{1}{m!}\sum_{p=1}^m (-1)^{m-p}\sum_{1\leq i_{1}<\cdots<i_{p}\leq m}\dfrac{\partial^m}{\partial r^m}f[r_{0};\xi_{i_{1}}+\hdots+\xi_{i_{p}}].\end{equation}
  To determine  higher order G\^ateaux derivatives it is more easy to use this equality.

The chain and product rules  and the Taylor expansion with integral remainder are still available for  $\mathscr{C}^m$-G\^ateaux differentiable maps (\cite{Schwartz} p. 30).   
  We use the following lemma to study the G\^ateaux differentiability of any applications mapping $r$ on the inverse of an element in a unitary topological algebra.

\begin{lemma}\label{d-1} Let $\mathcal{X}$ be a Fr\'echet space and $\mathcal{Y}$ be a unitary Fr\'echet algebra. Let $U$ be an open set of $\mathcal{X}$. Assume that  the application $f : U\rightarrow \mathcal{Y}$ is G\^ateaux differentiable at $r_{0}\in U$ and that $f(r)$ is invertible in $\mathcal{Y}$ for all $r\in U$ and that the application  $g : r\mapsto f(r)^{-1}$ is continuous at $r_{0}$. Then $g$ is G\^ateaux differentiable  at $r_{0}$ and its first derivative in the  direction $\xi\in\mathcal{X}$ is 

\begin{equation} \frac{\partial}{\partial r} f[r_{0},\xi] = -f(r_{0})^{-1} \circ \frac{\partial}{\partial r} f[r_{0},\xi] \circ f(r_{0})^{-1}.\end{equation}
 Moreover if $f$ is $\mathscr{C}^m$-G\^ateaux differentiable then $g$ is too.
\end{lemma}

\begin{proof} Let $\xi\in\mathcal{X}$ and $t>0$ small enough such that $(r_{0}+t\xi)\in U$, on a:
$$\begin{array}{rl}\hspace{-1mm}g(r_{0}+t\xi)-g(r_{0})=&f(r_{0})^{-1}\circ f(r_{0})\circ f(r_{0}+t\xi)^{-1}-f(r_{0})^{-1}\circ f(r_{0}+t\xi)\circ f(r_{0}+t\xi)^{-1}\vspace{2mm}\\=&f(r_{0})^{-1}\circ\left(f(r_{0})-f(r_{0}+t\xi)\right)\circ f(r_{0}+t\xi)^{-1}\vspace{2mm}\\=&f(r_{0})^{-1}\circ\left(f(r_{0})-f(r_{0}+t\xi)\right)\circ f(r_{0})^{-1}\vspace{2mm}\\&+f(r_{0})^{-1}\circ\left(f(r_{0})-f(r_{0}+t\xi)\right)\circ\left(f(r_{0}+t\xi)^{-1}-f(r_{0})^{-1}\right).\end{array}$$

 Since $g$ is continuous in $r_{0}$, we have $\lim\limits_{t\rightarrow0} \left(f(r_{0}+t\xi)^{-1}-f(r_{0})^{-1}\right)=0$ and since $f$ is G\^ateaux differentiable in $r_{0}$ we have $$\lim\limits_{t\rightarrow0}  \dfrac{f(r_{0})^{-1}\circ\left(f(r_{0})-f(r_{0}+t\xi)\right)\circ f(r_{0})^{-1}}{t}=-\left(f(r_{0})\right)^{-1}\circ \frac{\partial}{\partial r} f[r_{0},\xi]\circ\left(f(r_{0})\right)^{-1}.$$ As a consequence $$\lim_{t\rightarrow0}\frac{g(r_{0}+t\xi)-g(r_{0})}{t}=-\left(f(r_{0})\right)^{-1}\circ\frac{\partial}{\partial r} f[r_{0},\xi]\circ\left(f(r_{0})\right)^{-1}.$$
\end{proof}

%%%%%%%%%%%%%%%%%%%%
\section{G\^ateaux differentiability of pseudo-homogeneous kernels} \label{GDiffPH}
%+++++++++++++++++++++++++++++++++++
%\subsection{G\^ateaux differentiability properties}

Let $x_{r}$ denote an element of $\Gamma_{r}$ and let $\nn_{r}$ be the outer unit  normal vector to  $\Gamma_{r}$. When $r=0$ we write $\nn_{0}=\nn$. We note again $d\sigma$ the area element on $\Gamma_{r}$.

 In this  section we want to study the G\^ateaux differentiability of  the application mapping  $r\in B^{\infty}_{\epsilon}$ to the integral operator   $\mathcal{K}_{\Gamma_{r}}$ defined for a function $u_{r}\in H^{s}(\Gamma_{r})$ by:
\begin{equation}\mathcal{K}_{\Gamma_{r}}u_{r}(x_{r})=\mathrm{vp.}\int_{\Gamma_{r}}k_{r}(y_{r},x_{r}-y_{r})u_{r}(y_{r})d\sigma(y_{r}),\;x_{r}\in\Gamma_{r}\end{equation}
and of the  application mapping  $r\in B^{\infty}_{\epsilon}$ to the potential operator   $\mathcal{P}_{r}$  defined for a function $u_{r}\in H^{s}(\Gamma_{r})$ by:
\begin{equation}\mathcal{P}_{r}u_{r}(x)=\int_{\Gamma_{r}}k_{r}(y_{r},x-y_{r})u_{r}(y_{r})d\sigma(y_{r}),\;x\in K,\end{equation}
  where $k_{r}\in\mathscr{C}^{\infty}\left(\Gamma_{r}\times\left(\R^3\backslash\{0\}\right)\right)$ is a pseudo-homogeneous kernel of class $-m$ with $m\in\N$.  
  \smallskip
  
   We  want to differentiate  applications of the form $r\mapsto \mathcal{F}_{\Omega}(r)$ where the  domain of definition of $\mathcal{F}_{\Omega}(r)$ varies with $r$. How do we do? 
A first  idea, quite classical (see \cite{PierreHenrot, Potthast2, Potthast1}), is that  instead of  studying the application $$r\in B^{\infty}_{\epsilon}\mapsto \mathcal{F}_{\Omega}(r)\in\mathscr{C}^k(\Gamma_{r})$$ we consider the application $$r\in B^{\infty}_{\epsilon}\mapsto \mathcal{F}_{\Omega}(r)\circ(\Id+r)\in\mathscr{C}^k(\Gamma).$$ An example is $r\mapsto \nn_{r}$.
This point of view can be extended to Sobolev spaces $H^s(\Gamma)$, $s\in\R$.
 From now  we use  the transformation $\tau_{r}$  which maps a function $u_{r}$ defined on $\Gamma_{r}$ to the function $u_{r}\circ(\Id+r)$ defined on $\Gamma$. For all $r\in B^{\infty}_{\epsilon}$, this transformation $\tau_{r}$ admit an inverse. We have
$$(\tau_{r}u_{r})(x)=u_{r}(x+r(x))\text{ and }(\tau_{r}^{-1}u)(x_{r})=u(x).$$
 Then, instead of studying the application 
$$r\in B^{\infty}_{\epsilon}\mapsto \mathcal{K}_{\Gamma_{r}}\in\mathscr{L}_{c}\left(H^s(\Gamma_{r}),H^{s+m}(\Gamma_{r})\right)$$ we  consider the conjugate application
$$r\in B^{\infty}_{\epsilon}\mapsto \tau_{r}\mathcal{K}_{\Gamma_{r}}\tau_{r}^{-1}\in\mathscr{L}_{c}\left(H^s(\Gamma),H^{s+m}(\Gamma)\right).$$
In the framework of boundary integral equations, this approach is sufficient  to obtain the shape differentability of any solution to scalar boundary value problems \cite{Potthast2, Potthast1}.

Using  the change of variable  $x\mapsto x_{r}=x+r(x)$, we have for $u\in H^s(\Gamma)$:
\begin{equation*}\tau_{r}\mathcal{K}_{r}\tau_{r}^{-1}(u)(x)=\int_{\Gamma}k_{r}(y+r(y),x+r(x)-y-r(y))u(y)J_{r}(y)d\sigma(y),\;x\in\Gamma\end{equation*} where $J_{r}$ is the jacobian (the determinant of the Jacobian matrix) of the change of variable mapping $x\in\Gamma$ to $x+r(x)\in\Gamma_{r}$. The differentiablility analysis of these operators begins with the jacobian one. We have 
$$J_{r}=\Jac_{\Gamma}(\Id+r)=||\omega_{r}|| \text{ with }\omega_{r}=\com(\Id+\D r_{|\Gamma})\nn_{0}=\det(\Id+\D r_{|\Gamma})\transposee{(\Id+\D r_{|\Gamma})^{-1}}\nn,$$  
 and  the normal vector $\nn_{r}$ is given by
 $$\nn_{r}=\tau_{r}^{-1}\left(\frac{\omega_{r}}{\|\omega_{r}\|}\right).$$
The first derivative at $r=0$ of these applications are well known \cite{DelfourZolesio, PierreHenrot}. Here we present one method to obtain higher order derivative.

\begin{lemma}\label{J} The application $J$ mapping $r\in B^{\infty}_{\epsilon}$ to the jacobian $J_{r}\in\mathscr{C}^{\infty}(\Gamma,\R)$  is $\mathscr{C}^{\infty}$ G\^ateaux differentiable and its first derivative  at $r_{0}$ is  defined for $\xi\in\mathscr{C}^{\infty}_{b}(\R^3,\R^3)$ by:$$\frac{\partial J}{\partial r}[r_{0},\xi]=J_{r_{0}}(\tau_{r_{0}}\Div_{\Gamma_{r_{0}}}\tau^{-1}_{r_{0}})\xi.$$ 
\end{lemma}
\begin{proof} We just have to prove the $\mathscr{C}^{\infty}$-G\^ateaux differentiability of $W:r\mapsto w_{r}$. We do the proof for hypersurfaces $\Gamma$ of $\R^n$, $n\in\N$, $n\ge2$. We use local coordinate system. Assume that $\Gamma$ is parametrised by an atlas  $(\mathcal{O}_{i},\phi_{i})_{1\leq i\leq p}$  then $\Gamma_{r}$ can be parametrised by the atlas $(\mathcal{O}_{i},(\Id+r)\circ\phi_{i})_{1\leq i\leq p}$. For any $x\in\Gamma$, let us note $e_{1}(x),e_{2}(x),\hdots,e_{n-1}(x)$  the vector basis of the tangent plane to $\Gamma$ at $x$. The vector basis of the tangent plane to $\Gamma_{r}$ at $x+r(x)$ are given by 
$$e_{i}(r,x)=[(\Id+\D r)(x)]e_{i}(x)\quad\text{for }i=1,\hdots,n-1.$$ 
Thus, we have $\omega_{r}(x)=\dfrac{\bigwedge\limits_{i=1}^{n-1} e_{i}(r,x)}{\left|\bigwedge\limits_{i=1}^{n-1}e_{i}(x)\right|}$.
Since the applications $r\mapsto e_{i}(r,x)$, for $i=1,\hdots,n-1$ are $\mathscr{C}^{\infty}$-G\^ateaux differentiable, the application $W$ is too. Now want to compute the derivatives using the formula \eqref{Gsym}. Let $\xi\in\mathscr{C}^{\infty}_{b}(\R^n,\R^n)$ and $t$ small enough. We have at $r_{0}\in B_{\epsilon}^{\infty}$

$$\frac{\partial^m W}{\partial r^m}[r_{0},\xi]=\frac{\partial^m}{\partial t^m}_{\Big| t=0}\,\frac{\bigwedge\limits_{i=1}^{n-1}(\Id+Dr_{0}+tD\xi)e_{i}(x)}{\left|\bigwedge\limits_{i=1}^{n-1}e_{i}(x)\right|}.$$
To simplify this expression one have to note that $$\begin{array}{cl}[\D\xi(x)]e_{i}(x)&=[\D\xi(x)][(\Id+\D r_{0})(x)]^{-1}[(\Id+\D r_{0})(x)]e_{i}(x)\\&=[\D\xi(x)][\D(\Id+ r_{0})^{-1}(x+r_{0}(x))][(\Id+\D r_{0})(x)]e_{i}(x)\\&=[(\tau_{r_{0}}\D\tau_{r_{0}}^{-1})\xi(x)]e_{i}(r_{0},x)=[(\tau_{r_{0}}\D_{\Gamma_{r_{0}}}\tau_{r_{0}}^{-1})\xi(x)]e_{i}(r_{0},x).\end{array}$$
NB: given a $(n\times n)$ matrix $A$  we have $$\sum_{i=1}^{n-1}\cdots \times e_{i-1}\times A e_{i}\times e_{i+1}\times\cdots=(\Tr(A)\Id-\transposee{A})\bigwedge_{i=1}^{n-1}e_{i}.$$
 Thus we have with $A=[\tau_{r_{0}}\D_{\Gamma_{r_{0}}}\tau_{r_{0}}^{-1}\xi]$ and $B_{0}=\Id$, $B_{1}(A)=\Tr(A)\Id-\transposee{A}$
$$(\#)\left\{\begin{array}{ccl}W(r_{0})&=&J_{r_{0}}(\tau_{r_{0}}\nn_{r_{0}}),\vspace{2mm}\\\dfrac{\partial W}{\partial r}[r_{0},\xi]&=&J_{r_{0}}\left((\tau_{r_{0}}\Div_{\Gamma_{r_{0}}}\tau_{r_{0}}^{-1})\xi\cdot\tau_{r_{0}}\nn_{r_{0}}-[(\tau_{r_{0}}\nabla_{\Gamma_{r_{0}}}\tau_{r_{0}}^{-1})\xi]\tau_{r_{0}}\nn_{r_{0}}\right)\\&=&[B_{1}(A)\xi]W(r_{0}),\vspace{2mm}\\\dfrac{\partial^{m} W}{\partial r^{m}}[r_{0},\xi]&=&[B_{m}(A)\xi]W(r_{0})\\&=&\sum\limits_{i=1}^{m}(-1)^{i+1}\dfrac{(m-1)!}{(m-i)!}[B_{1}(A^{i})B_{m-i}(A)\xi]W(r_{0})\,\text{ for }m=1,\hdots,n-1,\vspace{2mm}\\\dfrac{\partial^{m} W}{\partial r^{m}}[r_{0},\xi]&\equiv&0\text{ for all }m\ge n.\end{array}\right.$$

It follows that 
$$\frac{\partial J}{\partial r}[r_{0},\xi]=\frac{1}{\|W(r_{0})\|}\frac{\partial W}{\partial r}[r_{0},\xi]\cdot W(r_{0})=\frac{\partial W}{\partial r}[r_{0},\xi]\cdot\tau_{r_{0}}\nn_{r_{0}}=J_{r_{0}}(\tau_{r_{0}}\Div_{\Gamma_{r_{0}}}\tau_{r_{0}}^{-1})\xi.$$

\end{proof}
Thanks to $(\#)$ we deduce  easily the G\^ateaux differentiability of $r\mapsto \tau_{r}\nn_{r}$. 
\begin{lemma} \label{N} The application $N$ mapping $r\in B^{\infty}_{\epsilon}$ to $\tau_{r}\nn_{r}=\nn_{r}\circ(\Id+r)\in\mathscr{C}^{\infty}(\Gamma,\R^3)$ is $\mathscr{C}^{\infty}$ G\^ateaux-differentiable and its first derivative  at $r_{0}$ is  defined for $\xi\in\mathscr{C}^{\infty}_{b}(\R^3,\R^3)$ by:  $$\frac{\partial N}{\partial r}[r_{0},\xi]=-\left[\tau_{r_{0}}\nabla_{\Gamma_{r_{0}}}\tau_{r_{0}}^{-1}\xi\right]N(r_{0}).$$ 
\end{lemma}
\begin{proof} This results from the precedent proof and we have
$$\begin{array}{ccl}\dfrac{\partial N}{\partial r}[r_{0},\xi]&=&\dfrac{1}{\|W(r_{0})\|}\dfrac{\partial W}{\partial r}[r_{0},\xi]-\dfrac{1}{\|W(r_{0})\|^3}\left(\dfrac{\partial W}{\partial r}[r_{0},\xi]\cdot W(r_{0})\right) W(r_{0})\vspace{2mm}\\&=&J_{r_{0}}^{-1}\left(\dfrac{\partial W}{\partial r}[r_{0},\xi]-\left(\dfrac{\partial W}{\partial r}[r_{0},\xi]\cdot(\tau_{r_{0}}\nn_{r_{0}})\right)\right)\tau_{r_{0}}\nn_{r_{0}}\vspace{2mm}\\&=&-\left[\tau_{r_{0}}\nabla_{\Gamma_{r_{0}}}\tau_{r_{0}}^{-1}\xi\right]\tau_{r_{0}}\nn_{r_{0}}.\end{array}$$
\end{proof}
To obtain higher order shape derivatives of these applications one can use the equalities $(\#)$ and 
$$(*)\left\{\begin{array}{ccl}\|\tau_{r}\nn_{r}\|&\equiv&1,\vspace{2mm}\\\dfrac{\partial^m N\cdot N}{\partial r^m}[r_{0},\xi]&\equiv&0\text{ for all }m\ge1.\end{array}\right.$$
As for exemple we have at $r=0$ in the direction $\xi\in\mathscr{C}^{\infty}_{b}(\R^3,\R^3)$:
$$\frac{\partial J}{\partial r}[0,\xi]=\Div_{\Gamma}\xi \text{ and }\frac{\partial N}{\partial r}[0,\xi]=-[\nabla_{\Gamma}\xi]\nn,$$
$$\frac{\partial^2 J}{\partial r^2}[0,\xi_{1},\xi_{2}]=-\Tr([\nabla_{\Gamma}\xi_{2}][\nabla_{\Gamma}\xi_{1}])+\Div_{\Gamma}\xi_{1}\cdot \Div_{\Gamma}\xi_{2}+\left([\nabla_{\Gamma}\xi_{1}]\nn\cdot[\nabla_{\Gamma}\xi_{2}]\nn\right).$$
Notice that $\Tr([\nabla_{\Gamma}\xi_{2}][\nabla_{\Gamma}\xi_{1}])=\Tr([\nabla_{\Gamma}\xi_{1}][\nabla_{\Gamma}\xi_{2}])$,
$$\frac{\partial^2 N}{\partial r^2}[0,\xi_{1},\xi_{2}]=[\nabla_{\Gamma}\xi_{2}][\nabla_{\Gamma}\xi_{1}]\nn+[\nabla_{\Gamma}\xi_{1}][\nabla_{\Gamma}\xi_{2}]\nn-\left([\nabla_{\Gamma}\xi_{1}]\nn\cdot[\nabla_{\Gamma}\xi_{2}]\nn\right)\nn.$$
For $n\ge3$ it needs too long calculations to simplify the expression of the derivatives and we only obtain the quadratic expression.  In the last section we give a second method to obtain higher order derivatives using the G\^ateaux derivatives of the surface differential operators.

\begin{remark} We do not need more than the first derivative of the deformations $\xi$. As a consequence for hypersurfaces of class $\mathscr{C}^{k+1}$, it suffice to consider deformations of class $\mathscr{C}^{k+1}$ to conserve the regularity of the jacobian and of the normal vector by differentiation.\end{remark}
\bigskip

   The following theorem  establish sufficient conditions for the G\^ateaux differentiability  of the boundary integral operators described here above and that we obtain their derivatives by deriving their kernels.
  \begin{theorem}\label{PDF} Let $k\in\N$. We set $(\Gamma\times\Gamma)^*=\left\{(x,y)\in\Gamma\times\Gamma; \;x\not=y\right\}$.
  Assume that
  \medskip
  
  1) For all fixed $(x, y)\in(\Gamma\times\Gamma)^*$ the function
  $$\begin{array}{cccl}f:&B^{\infty}_{\epsilon}&\rightarrow&\C\\&r&\mapsto& k_{r}(y+r(y),x+r(x)-y-r(y))J_{r}(y)\end{array}$$
is $\mathscr{C}^{k+1}$-G\^ateaux differentiable.\newline

2) The functions $(y,x-y)\mapsto f(r_{0})(y,x-y)$ and $$(y,x-y)\mapsto\dfrac{\partial^{\alpha}}{\partial r^{\alpha}}f[r_{0},\xi_{1},\hdots,\xi_{\alpha}](y,x-y)$$ are pseudo-homogeneous of class $-m$ for all $r_{0}\in B_{\epsilon}^{\infty}$, for all $\alpha=1,\hdots,k+1$ and for all $\xi_{1},\hdots,\xi_{k+1}\in\mathscr{C}^{\infty}_{b}(\R^3,\R^3)$.

Then the application $$\begin{array}{ccl}B_{\epsilon}^{\infty}&\rightarrow &\mathscr{L}_{c}(H^{s}(\Gamma), H^{s+m}(\Gamma))\\r&\mapsto&\tau_{r}\mathcal{K}_{\Gamma_{r}}\tau_{r}^{-1}\end{array}$$ is $\mathscr{C}^{k}$-G\^ateaux differentiable and
\begin{equation*}\frac{\partial^k }{\partial r^k}\left\{\tau_{r}\mathcal{K}_{\Gamma_{r}}\tau_{r}^{-1}\right\}[r_{0},\xi_{1},\hdots,\xi_{k}]u(x)=\int_{\Gamma}\frac{\partial^k}{\partial r^k}f[r_{0},\xi_{1},\hdots,\xi_{k}](y,x-y)u(y)d\sigma(y).\end{equation*}

\end{theorem}
\begin{proof}
We  use the  linearity of the integral and Taylor expansion with integral remainder. We do the proof   $k=1$ only. Let $r_{0}\in B_{\epsilon}^{\infty}$, $\xi\in\mathscr{C}^{\infty}(\R^n,\R^n)$ and $t$ small enough such that $r_{0}+t\xi\in B_{\epsilon}^{\infty}$. We have
$$f(r_{0}+t\xi,x,y)-f(r_{0},y,x-y)=t\frac{\partial f}{\partial r}[r_{0},\xi](y,x-y)+t^2\int_{0}^1(1-\lambda)\frac{\partial^2 f}{\partial r^2}[r_{0}+\lambda t\xi,\xi](y,x-y)d\lambda.$$
We have to verify that each terms in the equality here above is a kernel of an operator mapping $H^s(\Gamma)$ onto  $H^{s+m}(\Gamma)$.  The two first terms in the left hand side are kernels of class  $-m$ and by hypothesis  $\dfrac{\partial^2 f}{\partial r^2}[0,\xi]$ is also a kernel of class $-m$. It remains to prove that  the operator with kernel $$(x,y)\mapsto \int_{0}^1(1-\lambda)\frac{\partial^2 f}{\partial r^2}[r_{0}+\lambda t\xi,\xi](x,y)$$ acts from $H^s(\Gamma)$ to $H^{s+m}(\Gamma)$. Since $\dfrac{\partial^2 f}{\partial r^2}[r_{0}+\lambda t\xi,\xi]$ is pseudo-homogeneous of class  $-m$ for all $\lambda\in[0,1]$, it suffice to  use Lebesgue's theorem in order to invert the integration with respect to the variable $\lambda$ and the  integration with respect to $y$ on $\Gamma$. 
$$\begin{array}{ll}&\displaystyle{\left\|\int_{\Gamma}\left(\int_{0}^1(1-\lambda)\frac{\partial^2 f}{\partial r^2}[r_{0}+\lambda t\xi,\xi](x,y)d\lambda\right)u(y)d\sigma(y)\right\|_{H^{s+m}(\Gamma)}}\vspace{2mm}\\=&\displaystyle{\left\|\int_{0}^1(1-\lambda)\left(\int_{\Gamma}\frac{\partial^2 f}{\partial r^2}[r_{0}+\lambda t\xi,\xi](x,y)u(y)d\sigma(y)\right)d\lambda\right\|_{H^{s+m}(\Gamma)}}\vspace{2mm}\\\leq&\sup_{\lambda\in[0,1]}\displaystyle{\left\|\left(\int_{\Gamma}\frac{\partial^2 f}{\partial r^2}[r_{0}+\lambda t\xi,\xi](x,y)u(y)d\sigma(y)\right)\right\|_{H^{s+m}(\Gamma)}}\vspace{2mm}\\\leq&C||u||_{H^s(\Gamma)}.\end{array}$$

We then have \begin{equation*}\begin{split}&\dfrac{1}{t}\left(\int_{\Gamma}f(r_{0}+t\xi,x,y)u(y)d\sigma(y)-\int_{\Gamma}f(r_{0},x,y)u(y)d\sigma(y)\right)\\=&\int_{\Gamma}\frac{\partial f}{\partial r}[r_{0},\xi](x,y)u(y)d\sigma(y)+\;t\int_{\Gamma}\left(\int_{0}^1(1-\lambda)\frac{\partial^2 f}{\partial r^2}[r_{0}+\lambda t\xi,\xi](x,y)d\lambda\right)u(y)d\sigma(y).\end{split}\end{equation*}We pass to the  limit in $t=0$ and  we obtain the first G\^ateaux derivative. For higher order derivative it suffice to write the proof with $\dfrac{\partial^k}{\partial r^k}f[r_{0},\xi_{1},\hdots,\xi_{k}]$ instead of $f$. The linearity, the symetry and the continuity of the first derivative is deduced from the kernel one.
\end{proof}

Now we will consider some particular classes of pseudo-homogeneous kernels.
\begin{corollary} Assume that the kernels $k_{r}$ are of the form
$$k_{r}(y_{r},x_{r}-y_{r})=G(x_{r}-y_{r})$$
where $G$ is pseudo-homogeneous kernel which do not depend on $r$. Then the application $$\begin{array}{ccl}B_{\epsilon}^{\infty}&\rightarrow &\mathscr{L}_{c}(H^{s}(\Gamma), H^{s+m}(\Gamma))\\r&\mapsto&\tau_{r}\mathcal{K}_{\Gamma_{r}}\tau_{r}^{-1}\end{array}$$ is $\mathscr{C}^{\infty}$-G\^ateaux differentiable and
 the kernel of the first derivative at $r=0$ is defined for $\xi\in\mathscr{C}^{\infty}_{b}(\R^3,\R^3)$ by
\begin{equation*}\begin{split} \frac{\partial\left\{G(x+r(x)-y-r(y))\right\}}{\partial r}[0,\xi]=(\xi(x)-\xi(y))\cdot\nabla^zG(x-y)+G(x-y)\Div_{\Gamma}\xi(y).\end{split}\end{equation*}  
\end{corollary}
\begin{proof}  For all fixed $(x, y)\in(\Gamma\times\Gamma)^*$, consider the application 
$$f: U \mapsto f(r,x,y)=G(x+r(x)-y-r(y))J_{r}(y)\in\C.$$ 
We have to prove that $r\mapsto f(r)$ is $\mathscr{C}^{\infty}$-G\^ateaux differentiable and that each derivative define a pseudo-homogeneous kernel of class $-m$.

 $\rhd$\underline{Step 1:}\newline
 First of all we prove that for $(x,y)\in(\Gamma\times\Gamma)^*$ fixed the application $r\mapsto f(r,x,y)$ is infinitely G\^ateaux differentiable on $B^{\infty}_{\epsilon}$. By  lemma \ref{J} the application $r\mapsto J_{r}(y)$ is infinitely G\^ateaux differentiable on $B^{\infty}_{\epsilon}$, the application $r\mapsto x+r(x)$ is also infinitely G\^ateaux differentiable on $B^{\infty}_{\epsilon}$ and the kernel  $G$ is of class $\mathscr{C}^{\infty}$ on $\R^3\backslash\{0\}$. Being composed of applications infinitely G\^ateaux differentiable,  the application $r\mapsto f(r,x,y)$ is   too and using Leibniz  formula we have :
$$\begin{array}{l}\dfrac{\partial^k}{\partial r^k}f[r_{0},\xi_{1},\hdots,\xi_{k}](x,y)=\\\sum\limits_{\alpha=0}^k\sum\limits_{\sigma\in\mathcal{S}_{k}^+}\dfrac{\partial^\alpha}{\partial r^\alpha}\left\{G(x+r(x)-y-r(y))\right\}[r_{0},\xi_{\sigma(1)},\hdots,\xi_{\sigma(\alpha)}]\dfrac{\partial^{k-\alpha} J_{r}(y)}{\partial r^{k-\alpha}}[r_{0},\xi_{\sigma(\alpha+1)},\hdots,\xi_{\sigma(k)}]\end{array}$$ where $\mathcal{S}_{k}^+$ denote the non decreasing permutations of $\{1,\hdots,k\}$ and  $$\frac{\partial^\alpha}{\partial r^\alpha}\left\{G(x_{r}-y_{r})\right\}[r_{0};\xi_{1},\hdots,\xi_{\alpha}]=D^{\alpha}_{z}G[x+r_{0}(x)-y-r_{0}(y); \xi_{1}(x)-\xi_{1}(y),\hdots,\xi_{\alpha}(x)-\xi_{\alpha}(y)].$$

$\rhd$\underline{Step 2:}\newline
 We then prove that each derivative define a new  pseudo-homogeneous kernel of class $-m$  that is to say that for all $k\in\N$ and for all $k$-uple $(\xi_{1},\hdots,\xi_{k})$ the application
$$(x,y)\mapsto \frac{\partial^k}{\partial r^k}f[r_{0},\xi_{1},\hdots,\xi_{k}](x,y)$$ is pseudo-homogeneous of class $-m$.  Since  $\dfrac{\partial^{k-\alpha} J_{r}}{\partial r^{k-\alpha}}[r_{0},\xi_{1},\hdots,\xi_{k-\alpha}]\in\mathscr{C}^{\infty}(\Gamma,\R)$ we have to prove that  $$(x,y)\mapsto \frac{\partial^\alpha}{\partial r^\alpha}\left\{G(x+r(x)-y-r(y))\right\}[r_{0},\xi_{1},\hdots,\xi_{\alpha}]$$ defines a pseudo-homogeneous kernel of class  $-m$. 
 By definition, $G(z)$ admit the following asymptotic expansion when $z$ tends to zero:
$$G(z)=G_{m}(z)+\sum_{j=1}^{N-1}G_{m+j}(z)+G_{m+N}(h,z)$$ where $G_{m+j}$ is homogeneous of class $-(m+j)$ for $j=0,\hdots,N-1$ and $G_{m+N}$ is of arbitrary regularity. Using  Taylor formula we obtain the following result :
\begin{proposition} Let $G_{m}(z)$ be an homogeneous kernel of class $-m$ and any deformations $\xi=(\xi_{1},\hdots,\xi_{\alpha})\in\mathscr{C}^{\infty}_{b}(\R^3,\R^3)$. The function $$(x,y)\mapsto D^{\alpha}G_{m}[x+r_{0}(x)-y-r_{0}(y);\xi_{1}(x)-\xi_{1}(y),\hdots,\xi_{\alpha}(x)-\xi_{\alpha}(y)]$$ is pseudo-homogeneous of class $-m$.\end{proposition}
The application mapping $(\xi_{1},\hdots,\xi_{\alpha})\in\mathscr{C}^{\infty}_{b}(\R^3,\R^3)$ to the integral operator of kernel $$\dfrac{\partial^k\left\{G(x+r(x)-y-r(y))\right\}}{\partial r^m}[r_{0}, \xi_{1},\hdots,\xi_{k}]$$ is clearly linear and continuous for all $r_{0}\in B_{\epsilon}^{\infty}$.

\end{proof}

\begin{example}(\textbf{Single layer kernel})\label{Vk}  We  note $V_{\kappa}^{r}$ the integral operator defined for $u_{r}\in H^s(\Gamma_{r})$ by
$$V^{r}_{\kappa}u_{r}(x)=\int_{\Gamma_{r}}G(\kappa,|x_{r}-y_{r}|)u_{r}(y_{r})d\sigma(y_{r}).$$
The application
$$\begin{array}{ccl}B_{\delta}&\rightarrow &\mathscr{L}_{c}(H^{s}(\Gamma), H^{s+1}(\Gamma))\\r&\mapsto&\tau_{r}V^{r}_{\kappa}\tau_{r}^{-1}\end{array}$$ is $\mathscr{C}^{\infty}$-G\^ateaux differentiable and its  first derivative  at $r=0$  in the direction $\xi\in\mathscr{C}^{\infty}_{b}(\R^3,\R^3)$ is
\begin{equation}\frac{\partial \tau_{r} V^r_{\kappa}\tau_{r}^{-1}}{\partial r}[0,\xi]u(x)=\int_{\Gamma}k'(y,x-y)u(y)d\sigma(y)\end{equation}
where in $\R^3$ we have
\begin{equation*}\begin{split} k'(x,y)=&G(\kappa,|x-y|)\left(\frac{(\xi(x)-\xi(y))\cdot(x-y)}{|x-y|}\left(i\kappa-\frac{1}{|x-y|}\right)+\Div_{\Gamma}\xi(y)\right).\end{split}\end{equation*}  
\end{example}

\begin{example}(\textbf{Double layer kernel})\label{Dk}
 We note $D_{\kappa}^{r}$ the integral operator defined for $u_{r}\in H^s(\Gamma_{r})$ by
$$D^{r}_{\kappa}u_{r}(x)=\int_{\Gamma_{r}}\nn_{r}(x_{r})\cdot\nabla^zG(\kappa,|x_{r}-y_{r}|)u_{r}(y_{r})d\sigma(y_{r}).$$
The application
$$\begin{array}{ccl}B_{\delta}&\rightarrow &\mathscr{L}_{c}(H^{s}(\Gamma), H^{s+1}(\Gamma))\\r&\mapsto&\tau_{r}D^{r}_{\kappa}\tau_{r}^{-1}\end{array}$$ is $\mathscr{C}^{\infty}$ G\^ateaux-differentiable . 
\end{example}
\begin{proof} 
We have $$\nn_{r}(x_{r})\cdot\nabla^zG(\kappa,|x_{r}-y_{r}|)u_{r}(y_{r})=\nn_{r}(x_{r})\cdot(x_{r}-y_{r})\frac{G(\kappa,|x_{r}-y_{r}|)}{|x_{r}-y_{r}|}\left(i\kappa-\frac{1}{|x_{r}-y_{r}|}\right).$$ 
We have to prove that $$r\in B_{\epsilon}^{\infty}\mapsto\frac{(\tau_{r}\nn_{r})(x)\cdot(x+r(x)-y-r(y))}{|x+r(x)-y-r(y)|^3}$$ is $\mathscr{C}^{\infty}$ G\^ateaux differentiable and that the derivatives are pseudo-homogeneous of class $-1$. To do so we use local coordinates as Potthast did in \cite{Potthast1} and prove that 

$$\frac{\partial^k \;(\tau_{r}\nn_{r})(x)\cdot(x+r(x)-y-r(y))}{\partial r^k}[r_{0},\xi_{1},\hdots,\xi_{k}]$$ behaves as $|x-y|^2$ when $x-y$ tends to zero.
\end{proof}
\medskip

Each domain $\Omega$ is a countable union of compact subset of $\Omega$: $\Omega=\bigcup\limits_{p\ge1}K_{p}$.  Instead of studying the application
$$r\in B^{\infty}_{\epsilon}\mapsto \mathcal{F}_{\Omega}(r)\in\mathscr{L}_{c}\left(H^s(\Gamma_{r}),H^{s+m}(\Omega_{r})\right)$$ we  consider the  application
$$r\in B^{\infty}_{\epsilon}\mapsto \mathcal{F}_{\Omega}(r)\tau_{r}^{-1}\in\mathscr{L}_{c}\left(H^s(\Gamma),H^{s+m}(K_{p})\right).$$ We use this approach for potential operators. We have for $u\in H^{s-\frac{1}{2}}(\Gamma)$

\begin{equation*}\mathcal{P}_{r}\tau_{r}^{-1}(u)(x)=\int_{\Gamma}k_{r}(y+r(y),x-y-r(y))u(y)J_{r}(y)d\sigma(y),\;x\in K_{p}.\end{equation*}

\begin{theorem}\label{P'DF}Let $s\in\R$. Let $G(z)$ be a  pseudo-homogeneous kernel of class $-(m+1)$ with $m\in\N$. Assume that for all  $r\in B_{\epsilon}^{\infty}$, we have $k_{r}(y_{r},x-y_{r})=G(x-y_{r})$. Then the application
$$\begin{array}{ccl}B_{\epsilon}^{\infty}&\rightarrow&\mathscr{L}_{c}\left(H^{s-\frac{1}{2}}(\Gamma), \mathscr{C}^{\infty}(K_{p})\right)\\r&\mapsto&\mathcal{P}_{r}\tau_{r}^{-1}\end{array}$$ is infinitely G\^ateaux differentiable and
\begin{equation*}\frac{\partial^k \mathcal{P}_{r}\tau_{r}^{-1}}{\partial r^k}[r_{0},\xi_{1},\hdots,\xi_{k}]u(x)=\int_{\Gamma}\frac{\partial^k}{\partial r^k}\left\{G(x-y-r(y))J_{r}(y)\right\}[r_{0},\xi_{1},\hdots,\xi_{k}]u(y)d\sigma(y).\end{equation*} Its first derivative at $r=0$ in the  direction $\xi\in\mathscr{C}^{\infty}_{b}(\R^3,\R^3)$ is the integral operator  denoted by $\mathcal{P}^{1}$ with the kernel \begin{equation}\frac{\partial}{\partial r}\left\{G(x-y-r(y))\right\}[r_{0},\xi]=-\xi(y)\cdot\nabla^zG(x-y)+G(x-y)\Div_{\Gamma}\xi(y).\end{equation}The operator $\mathcal{P}^{(1)}$ can be extended in a linear and continuous integral  operator from $H^{s-\frac{1}{2}}(\Gamma)$ to $H^{s+m}(\Omega)$ and $H_{loc}^{s+m}(\Omega).$ \end{theorem}

\begin{proof} The kernel and its higher order derivatives are of class  $\mathscr{C}^{\infty}$ on $K_{p}$. \newline 
Since $\Omega$ is an increasing union of compact manifolds we can define a shape derivative on the whole domain $\Omega$.  Let us  look at the first derivative : the term $G(x-y)\Div_{\Gamma}\xi(y)$ has the same regularity than $G(x-y)$ when $x-y$ tends to zero wheareas  $\xi(y)\cdot\nabla G(x-y)$ loose one order of regularity.  As a consequence the kernel must be of class $-m+1$ in order that its first derivative acts from $H^{s-\frac{1}{2}}(\Gamma)$ to $H^{s+m}(\Omega)$ and $H^{s+m}_{loc}(\Omega^c)$. \end{proof}

\begin{remark} We conclude that the boundary integral operators are smooth with respect to the domain whereas the potential operators loose one order of regularity at each derivation. We point out that we do not need more than the first derivative of the deformations $\xi$ to compute the G\^ateaux derivatives of these integral operators.
\end{remark}

\begin{example}(\textbf{Single layer potential})\label{psi}
 We denote by $\psi_{\kappa}^{r}$ the single layer potential  defined for $u_{r}\in H^s(\Gamma_{r})$ by
$$\psi^{r}_{\kappa}u_{r}(x)=\int_{\Gamma_{r}}G(\kappa,|x-y_{r}|)u_{r}(y_{r})d\sigma(y_{r}),\;x\in\R^3\backslash\Gamma_{r}.$$
 The application
$$\begin{array}{ccl}B_{\epsilon}^{\infty}&\rightarrow &\mathscr{L}_{c}\left(H^{s}(\Gamma), \mathscr{C}^{\infty}(K_{p})\right)\\r&\mapsto&\tau_{r}\psi^{r}_{\kappa}\tau_{r}^{-1}\end{array}$$ is infinitely  G\^ateaux differentiable. Its first derivative at $r=0$ can be extended in a linear and continuous operator from   $H^{s-\frac{1}{2}}(\Gamma)$ to $H^s(\Omega)\cup H^s_{loc}(\Omega^c)$. 
\end{example}

Since the potential operators are infinitely G\^ateaux differentiable far from the boundary we have the following result by inverting the derivation with respect to $r$ and the passage to the limit $|x|\rightarrow\infty$.

\begin{example} Let $s\in\R$.
We denote by $\psi_{\kappa}^{\infty,r}$ the far field operator associated to the  single layer potential  defined for $u_{r}\in H^s(\Gamma_{r})$ by
$$\psi^{\infty,r}_{\kappa}u_{r}(\hat{x})=\int_{\Gamma_{r}}e^{-i\kappa\hat{x}\cdot y_{r}}u_{r}(y_{r})d\sigma(y_{r}),\;\hat{x}\in S^2.$$
  The application $$\begin{array}{lcl}B^{\infty}_{\epsilon}&\rightarrow&\mathscr{L}_{c}(H^{s}(\Gamma),\mathscr{C}^{\infty}(S^2))\\r&\mapsto&\Psi_{\kappa}^{\infty,r}\tau_{r}^{-1}\end{array}$$ is infinitely G\^ateaux differentiable and its first deriavtive at $r=0$ is defined  for $u\in H^{s}(\Gamma)$ by:
$$\frac{\partial \Psi_{{\kappa}}^{\infty,r}\tau_{r}^{-1} }{\partial r}[0,\xi]u(\hat{x})=\left(\int_{\Gamma}e^{-i\kappa\hat{x}\cdot y}\left(\Div_{\Gamma}\xi(y)-i\kappa\hat{x}\cdot\xi(y)\right)u(y)d\sigma(y)\right).$$ 
\end{example}

%%%%%%%%%%%%%%%%%%%%%%%%%%%%%%
\section{Shape differentiability of the solution}  \label{ShapeSol}
 Let $\EE^{inc}$ be  an incident electric field which is a fixed data.  The aim of this section is to study the shape differentiation properties of the application $\EE$ mapping the bounded scatterer $\Omega$  to  the solution $\EE(\Omega)=\left(\EE^i(\Omega),\EE^s(\Omega)\right)\in\HH_{loc}(\Rot,\R^3)$ to the dielectric scattering problem by the obstacle $\Omega$ lit by  the incident field  $\EE^{inc}$ established in section \ref{ScatProb}.  To do so we use the integral representation of  the solution.
 
 We set $\mathscr{E}^{i}(r)=\EE^{i}(\Omega_{r})$ and $\mathscr{E}^s(r)=\EE^s(\Omega_{r})$ and we denote  $\Psi_{E_{\kappa}}^{r}$, $\Psi_{M_{\kappa}}^{r}$, $C_{0}^{*r}$, $C_{\kappa}^{r}$ and $M_{\kappa}^{r}$ the potential operators and the boundary integral operators on $\Gamma_{r}$ and $\gamma_{D}^r$, $\gamma_{N_{\kappa}}^r$, $\gamma_{D}^{c,r}$ the $\gamma_{N_{\kappa}}^{c,r}$ trace mappings on $\Gamma_{r}$. We have :
 
  \begin{equation}\label{risr1}\mathscr{E}^{tot}(r)=\EE^{inc}+\mathscr{E}^s(r)\end{equation}with
\begin{equation} \mathscr{E}^s(r)= \left(-\Psi_{E_{\kappa_{e}}}^{r}-i\eta\Psi_{M_{\kappa_{e}}}^{r}C_{0}^{*r}\right)\jj_{r}\qquad\text{ dans }\Omega_{r}^c=\R^3\backslash\overline{\Omega_{r}}\end{equation}  where $\jj_{r}$ solves the integral equation $$\SS^{r}\jj_{r}=-\rho\left(-\frac{1}{2}\Id+M_{\kappa_{i}}^{r}\right)\gamma_{D}^{r}\EE^{inc}-C_{\kappa_{i}}^{r}\gamma_{N_{\kappa_{e}}}^{r}\EE^{inc},$$
and
 \begin{equation}\label{risr2} \mathscr{E}^{i}(r) = -\frac{1}{\rho}\Psi_{E_{\kappa_{i}}}^{r}\gamma_{N_{\kappa_{e}}}^{c,r}\mathscr{E}^{tot}(r) - \Psi_{M_{\kappa_{i}}}^{r}\gamma_{D}^{c,r}\mathscr{E}^{tot}(r)\qquad\text{ dans }\Omega_{r}\end{equation}
 Recall that the operator  $\SS^{r}$  is composed of the operators  $C_{\kappa_{e}}^{r}$, $M_{\kappa_{e}}^{r}$, $C_{\kappa_{i}}^{r}$ et $M_{\kappa_{i}}^{r}$ and that these last ones are defined on the space $\TT\HH^{-\frac{1}{2}}(\Div_{\Gamma_{r}},\Gamma_{r})$.
 
%++++++++++++++++++++++++++++++++++++++
\subsection{Variations of Helmholtz decomposition}
 We have to turn out many difficulties. On one hand,  to be able to construct shape derivatives of the solution it is necessary to prove that the derivatives are defined on the same spaces than the boundary
 integral operators themselves, that is  $\TT\HH^{-\frac{1}{2}}(\Div_{\Gamma},\Gamma)$ (if we derive at $r=0$). On the other hand, the very definition of the differentiability of operators defined on $\TT\HH\sp{-\frac{1}{2}}(\Div_{\Gamma},\Gamma)$ raises non-trivial questions. The first one is : {\bf How to derive applications defined on the variable space  $\TT\HH^{-\frac{1}{2}}(\Div_{\Gamma_{r}},\Gamma_{r})$?}\newline
A first  idea is to insert the identity  $\tau_{r}^{-1}\tau_{r}=\Id_{\HH^{-\frac{1}{2}}(\Gamma_{r})}$ between each operator in the integral  representation  of the  solution in order to consider integral operators on the fixed  boundary $\Gamma$ only and to study the differentiability of the applications
$$\begin{array}{ccl} r&\mapsto&\tau_{r}C_{\kappa}^{r}\tau_{r}^{-1},\\ r&\mapsto&\tau_{r}M_{\kappa}^{r}\tau_{r}^{-1},\\ r&\mapsto&\Psi_{E_{\kappa}}^{r}\tau_{r}^{-1},\\ r&\mapsto&\Psi_{M_{\kappa}}^{r}\tau_{r}^{-1}\end{array}$$ but many difficulties persist as Potthast pointed out \cite{Potthast3}.  The electromagnetic boundary integral operators are defined and bounded on tangential functions to  $\Gamma_{r}$.The restriction of the operator $\tau_{r}M_{\kappa}^r\tau_{r}^{-1}$ to tangential densities to $\Gamma_{r}$, has the same regularity of the double layer  potential operator. If we differentiate $\tau_{r}M_{\kappa}^r\tau_{r}^{-1}$, we will not obtain an  operator with the same  regularity than $M_{\kappa}$ and acting on $\TT\HH^{-\frac{1}{2}}(\Div_{\Gamma},\Gamma)$ since: $$\tau_{r}(\TT\HH^{-\frac{1}{2}}(\Div_{\Gamma_{r}},\Gamma_{r}))\not=\TT\HH^{-\frac{1}{2}}(\Div_{\Gamma},\Gamma).$$
 The incident field  $\EE^{inc}$ is analytic in the neighborhood of $\Gamma$ thus $\gamma_{D}^{r} \EE^{inc}\in \TT\HH^{-\frac{1}{2}}(\Div_{\Gamma_{r}},\Gamma_{r})$ for all $r\in B^{\infty}_{\epsilon}$.  Set $f(r)=\tau_{r}\left(\gamma_{D}^r\EE^{inc}\right)$. For $\xi\in\mathscr{C}^{\infty}_{b}(\R^3,\R^3)$, the G\^ateaux semiderivative $\dfrac{\partial f(t\xi)}{\partial t}_{|t=0}$  is not tangent to $ \Gamma$ anymore it follows that  $M_{\kappa}\dfrac{\partial f(t\xi)}{\partial t}_{|t=0}$ is not defined. We have  the same difficulties we the Neuman trace $\gamma_{N_{\kappa}}^r$ and the  other operators.

As an alternative, the idea of R. Potthast was to introduce projectors on the  tangent planes of the surfaces  $\Gamma$ and $\Gamma_{r}$. Let us note $\pi(r)$ the orthogonal projection of any functions defined on $\Gamma_{r}$ onto the tangent plane to $\Gamma$. This is a linear and continuous operator from the  continous vector function space on $\Gamma_{r}$ to the the space of continuous tangential function  to $\Gamma$ and for $\uu_{r}\in\left(\mathscr{C}(\Gamma_{r})\right)^3$ we have $$(\pi(r)\uu_{r})(x)=\uu_{r}(x+r(x))-\left(\nn(x)\cdot \uu_{r}(x+r(x))\right)\nn(x).$$ 

\begin{proposition} The restriction of $\pi(r)$ to the continuous and tangential functions to $\Gamma_{r}$ admit an inverse, denoted by $\pi^{-1}(r)$. The application $\pi^{-1}(r)$ is defined for a tangential function $\uu$ to $\Gamma$ by 
$$(\pi^{-1}(r)\uu)(x+r(x))=\uu(x)-\nn(x)\frac{\nn_{r}(x+r(x))\cdot \uu(x)}{\nn_{r}(x+r(x))\cdot \nn(x)}.$$ 
And we have  $\pi^{-1}(r)\uu\in \TT\HH^s(\Gamma_{r})$ if and only if  $\uu\in \TT\HH^s(\Gamma)$.
\end{proposition}
 In the framework of the space of tangential continuous functions it  suffices to insert
the product $\pi^{-1}(r)\pi(r)=\Id_{\TT\mathscr{C}^0(\Gamma_{r})}$ in the integral representation of the  solution to lead us  to study boundary integral operators defined on $\TT\mathscr{C}^0(\Gamma)$ which do not  depend  on $r$ anymore but here  we would obtain operators defined on $$\pi(r)\left(\TT\HH^{-\frac{1}{2}}(\Div_{\Gamma_{r}},\Gamma_{r})\right)=\left\{u\in \TT\HH^{-\frac{1}{2}}(\Gamma), \Div_{\Gamma_{r}}(\pi^{-1}(r)\uu)\in H^{-\frac{1}{2}}(\Gamma_{r})\right\}.$$
This space depends again on the variable $r$ and do not correspond to $\TT\HH^{-\frac{1}{2}}(\Div_{\Gamma},\Gamma)$.
Our approach consist in using the Helmholtz decomposition of the spaces $\TT\HH^{-\frac{1}{2}}(\Div_{\Gamma_{r}},\Gamma_{r})$ for $r\in B^{\infty}_{\epsilon}$ and to  introduce a new   invertible operator $\Pp_{r}$ defined on $\TT\HH^{-\frac{1}{2}}(\Div_{\Gamma_{r}},\Gamma_{r})$ and which is not a projection operator. 
 \bigskip
 
 We have the following decomposition. We refer to \cite{delaBourdonnaye} for the proof.
 \begin{theorem} The Hilbert space $\TT\HH^{-\frac{1}{2}}(\Div_{\Gamma},\Gamma)$ admit the following Helmholtz decomposition:\begin{equation} \TT\HH^{-\frac{1}{2}}(\Div_{\Gamma},\Gamma)= \nabla_{\Gamma}\left(H^{\frac{3}{2}}(\Gamma)\slash\R \right)\bigoplus {\Rot}_{\Gamma}\left(H^{\frac{1}{2}}(\Gamma)\slash\R\right).\end{equation}
 \end{theorem} 
   Since the  real $\epsilon$ is chosen such that for all $r\in B^{\infty}_{\epsilon}$ the surfaces  $\Gamma_{r}$ are still regular and simply connected, then the spaces $\TT\HH^{-\frac{1}{2}}(\Div_{\Gamma_{r}},\Gamma_{r})$ admit the same decomposition.\newline
 Let $\jj_{r}\in \TT\HH^{-\frac{1}{2}}(\Div_{\Gamma_{r}},\Gamma_{r})$ and let $\nabla_{\Gamma_{r}}\;p_{r}+\Rot_{\Gamma_{r}}\;q_{r}$ its Helmholtz decomposition. Since $p_{r}\in H^{\frac{3}{2}}(\Gamma_{r})$ and $q_{r}\in H^{\frac{1}{2}}(\Gamma_{r})$, their change of variables from $\Gamma_{r}$ to $\Gamma$, $\tau_{r}(p_{r})$ and $\tau_{r}(q_{r})$, are in  $H^{\frac{3}{2}}(\Gamma)$ and $H^{\frac{1}{2}}(\Gamma)$ respectively. The following operator :   
 
 \begin{equation}\begin{array}{llcl} \Pp_{r}:&\TT\HH^{-\frac{1}{2}}(\Div_{\Gamma_{r}},\Gamma_{r})&\longrightarrow& \TT\HH^{-\frac{1}{2}}(\Div_{\Gamma},\Gamma)\\&\jj_{r}=\nabla_{\Gamma_{r}}\;p_{r}+\mathbf{rot}_{\Gamma_{r}}\;q_{r}&\mapsto&\jj=\nabla_{\Gamma}\;\tau_{r}p_{r}+\mathbf{rot}_{\Gamma}\;\tau_{r}q_{r}\end{array}\end{equation}is well-defined.
 
 The operator $\Pp_{r}$ transforms a tangential vector field  $\jj_{r}$ to $\Gamma_{r}$ in a tangential vector field $\jj$ to $\Gamma$. This operator is linear, continuous and admit an inverse $\Pp_{r}^{-1}$  given by :
 
 \begin{equation} \begin{array}{llcl}\Pp_{r}^{-1}:&\TT\HH^{-\frac{1}{2}}(\Div_{\Gamma},\Gamma)&\longrightarrow &\TT\HH^{-\frac{1}{2}}(\Div_{\Gamma_{r}},\Gamma_{r})\\&\jj=\nabla_{\Gamma}\;p+\mathbf{rot}_{\Gamma}\;q&\mapsto&\jj_{r}= \nabla_{\Gamma_{r}}\;\tau_{r}^{-1}(p)+\mathbf{rot}_{\Gamma_{r}}\;\tau_{r}^{-1}(q).\end{array}\end{equation}
 Obviously we have when $r=0$ that $\Pp_{r}=\Pp_{r}^{-1}=\Id_{\TT\HH^{-\frac{1}{2}}(\Div_{\Gamma},\Gamma)}$. We insert the  identity $\Id_{\TT\HH^{-\frac{1}{2}}(\Div_{\Gamma_{r}},\Gamma_{r})}=\Pp_{r}^{-1}\Pp_{r}$ between each operator in the integral  representation  of the  solution  $(\mathscr{E}^{i}(r),\mathscr{E}^s(r))$. 
 Finaly we have to study the G\^ateaux differentiality properties of the following applications :
 
 \begin{equation}\label{reformulation}\begin{array}{lclll}
 \;  B^{\infty}_{\epsilon} &\rightarrow& \mathscr{L}_{c}(\TT\HH^{-\frac{1}{2}}(\Div_{\Gamma},\Gamma), \HH(\Rot, K_{p})) & : & r\mapsto \Psi^{r}_{E_{\kappa}}\Pp_{r}^{-1} \\ 
 \; B^{\infty}_{\epsilon}& \rightarrow& \mathscr{L}_{c}(\TT\HH^{-\frac{1}{2}}(\Div_{\Gamma},\Gamma), \HH(\Rot,K_{p})) & : & r\mapsto \Psi^{r}_{M_{\kappa}}\Pp_{r}^{-1}\\
 \;   B^{\infty}_{\epsilon}& \rightarrow& \mathscr{L}_{c}(\TT\HH^{s}(\Div_{\Gamma},\Gamma), \TT\HH^{-\frac{1}{2}}(\Div_{\Gamma},\Gamma))& : & r\mapsto \Pp_{r}M^{r}_{\kappa}\Pp_{r}^{-1}\\
  \;B^{\infty}_{\epsilon}&\rightarrow& \mathscr{L}_{c}(\TT\HH^{-\frac{1}{2}}(\Div_{\Gamma},\Gamma), \TT\HH^{-\frac{1}{2}}(\Div_{\Gamma},\Gamma))& : & r\mapsto \Pp_{r}C^{r}_{\kappa}\Pp_{r}^{-1}\\\end{array}\end{equation}
  where $K_{p}$ is a compact subset of $\R^3\backslash\Gamma$.
  \bigskip
  
  Now let us look at the  integral representation of these  operators .

$\rhd$\textbf{ Integral representation of $\Psi^{r}_{E_{\kappa}}\Pp_{r}^{-1}$ .}\newline
 The operator $\Psi^{r}_{E_{\kappa}}\mathbf{P_{r}}^{-1}$  is defined for $\jj=\nabla_{\Gamma}\;p+\mathbf{\Rot}_{\Gamma}\;q\in \TT\HH^{-\frac{1}{2}}(\Div_{\Gamma},\Gamma)$ and $x\in K_{p}$ by:  
\begin{equation*}\label{PE}\begin{split}\Psi^{r}_{E_{\kappa}}\Pp_{r}^{-1}\jj(x) =&\;\kappa\displaystyle{\int_{\Gamma_{r}}G(\kappa,|x-y_{r}|)\left(\nabla_{\Gamma_{r}}\tau_{r}^{-1}p\right)(y_{r})d\sigma(y_{r})}\\&\;+\kappa\displaystyle{\int_{\Gamma_{r}}G(\kappa,|x-y_{r}|)\left(\mathbf{rot}_{\Gamma_{r}} \tau_{r}^{-1}q\right)(y_{r})d\sigma(y_{r})} \\ &\;- \kappa^{-1}\nabla\displaystyle{\int_{\Gamma_{r}}G(\kappa,|x-y_{r}|)\left(\Delta_{\Gamma_{r}}\tau_{r}^{-1} p\right)(y_{r})d\sigma(y_{r})}.\end{split}\end{equation*}
\medskip

$\rhd$\textbf{ Integral representation of  $\Psi^{r}_{M_{\kappa}}\Pp_{r}^{-1}$.}\newline
 The operator $\Psi^{r}_{M_{\kappa}}\mathbf{P_{r}}^{-1}$  is defined for $\jj=\nabla_{\Gamma}\;p+\mathbf{\Rot}_{\Gamma}\;q\in \TT\HH^{-\frac{1}{2}}(\Div_{\Gamma},\Gamma)$ and $x\in K_{p}$ by:   
\begin{equation*}\label{PM}\begin{split}\Psi^{r}_{M_{\kappa}}\Pp_{r}^{-1}\jj(x)=&\;\Rot\displaystyle{\int_{\Gamma_{r}}G(\kappa,|x-y_{r}|)\left(\nabla_{\Gamma_{r}}\tau_{r}^{-1}p\right)(y_{r})d\sigma(y_{r})}\\&+\Rot\displaystyle{\int_{\Gamma_{r}}G(\kappa,|x-y_{r}|)\left(\mathbf{\Rot}_{\Gamma_{r}}\tau_{r}^{-1} q\right)(y_{r})(y_{r})d\sigma(y_{r})}.\end{split}\end{equation*}
\medskip

$\rhd$\textbf{ Integral representation of $\Pp_{r}C^{r}_{\kappa}\Pp_{r}^{-1}$ .}\newline Recall that for $\jj_{r}\in\TT\HH^{-\frac{1}{2}}(\Div_{\Gamma_{r}},\Gamma_{r})$, the operator $C_{\kappa}^r$ is defined by
\begin{equation*}\begin{array}{ll}C^{r}_{\kappa}\jj_{r}(x_{r})=&-\kappa\,\nn_{r}(x_{r})\times\displaystyle{\int_{\Gamma_{r}}G(\kappa,|x_{r}-y_{r}|)\jj_{r}(y_{r})d\sigma(y_{r})}\vspace{2mm}\\&-\kappa^{-1}\nn_{r}(x_{r})\times\nabla_{\Gamma_{r}}^{x_{r}}\displaystyle{\int_{\Gamma_{r}}G(\kappa,|x_{r}-y_{r}|)\Div_{\Gamma_{r}}\jj_{r}(y_{r})d\sigma(y_{r})}.\end{array}\end{equation*} We want to write  $\operatorname{C^{r}_{\kappa}}\jj_{r}$ of the form $\nabla_{\Gamma_{r}}P_{r}+\Rot_{\Gamma_{r}}Q_{r}$. 
Using the formula \eqref{eqd2}-\eqref{eqd3} we deduce that :
$$ \Div_{\Gamma_{r}}\operatorname{C^{r}_{\kappa}}\jj_{r}=\Delta_{\Gamma_{r}} P_{r}\;\text{ et } \rot_{\Gamma_{r}}\operatorname{C^{r}_{\kappa}}\jj_{r}=-\Delta_{\Gamma_{r}} Q_{r}.$$
 As a consequence we have for $x_{r}\in\Gamma_{r}$:
\begin{equation}\begin{array}{ll}P_{r}(x_{r})=&-\kappa\;\Delta_{\Gamma_{r}}^{-1}\Div_{\Gamma_{r}}\left(\nn_{r}(x_{r})\times\displaystyle{\int_{\Gamma_{r}}G(\kappa,|x_{r}-y_{r}|)\jj_{r}(y_{r})d\sigma(y_{r})}\right)\end{array}\end{equation}and
\begin{equation*}\begin{array}{rl}Q_{r}(x_{r})=&-\kappa\;(-\Delta_{\Gamma_{r}}^{-1})\rot_{\Gamma_{r}}\left(\nn_{r}(x_{r})\times\displaystyle{\int_{\Gamma_{r}}G(\kappa,|x_{r}-y_{r}|)\jj_{r}(y_{r})d\sigma(y_{r})}\right)\vspace{2mm}\\&-\kappa^{-1}(-\Delta_{\Gamma_{r}})\rot_{\Gamma_{r}}(-\Rot_{\Gamma_{r}})\displaystyle{\int_{\Gamma_{r}}G(\kappa,|x_{r}-y_{r}|)\Div_{\Gamma_{r}}\jj_{r}(y_{r})d\sigma(y_{r})},\\=&\kappa\;\Delta_{\Gamma_{r}}^{-1}\rot_{\Gamma_{r}}\left(\nn_{r}(x_{r})\times\displaystyle{\int_{\Gamma_{r}}G(\kappa,|x_{r}-y_{r}|)\jj_{r}(y_{r})d\sigma(y_{r})}\right)\vspace{2mm}\\&+\kappa^{-1}\displaystyle{\int_{\Gamma_{r}}G(\kappa,|x_{r}-y_{r}|)\Div_{\Gamma_{r}}\jj_{r}(y_{r})d\sigma(y_{r})}.\end{array}\end{equation*}
 The operator $\Pp_{r}C^{r}_{\kappa}\Pp_{r}^{-1}$ is defined for $\jj=\nabla_{\Gamma}\;p+\mathbf{\Rot}_{\Gamma}\;q\in \TT\HH^{-\frac{1}{2}}(\Div_{\Gamma},\Gamma)$ by: 
$$\Pp_{r}C^{r}_{\kappa}\Pp_{r}^{-1}=\nabla_{\Gamma}P(r)+\Rot_{\Gamma}Q(r),$$ with

\begin{equation*}\begin{array}{ll}P(r)(x)=&-\kappa\;\left(\tau_{r}\Delta_{\Gamma_{r}}^{-1}\Div_{\Gamma_{r}}\tau_{r}^{-1}\right)\left((\tau_{r}\nn_{r})(x)\times\tau_{r}\left\{\displaystyle{\int_{\Gamma_{r}}G(\kappa,|\cdot-y_{r}|)(\nabla_{\Gamma_{r}}\tau_{r}^{-1}p)(y_{r})d\sigma(y_{r})}\right.\right.\vspace{2mm}\\&\hspace{5.3cm}\left.\left.+\displaystyle{\int_{\Gamma_{r}}G(\kappa,|\cdot-y_{r}|)(\mathbf{rot}_{\Gamma_{r}}\tau_{r}^{-1}q)(y_{r})d\sigma(y_{r})}\right\}(x)\right)\end{array}\end{equation*}and

\begin{equation*}\begin{array}{ll}Q(r)(x)=&\kappa\;\left(\tau_{r}\Delta_{\Gamma_{r}}^{-1}\rot_{\Gamma_{r}}\tau_{r}^{-1}\right)\left((\tau_{r}\nn_{r})(x)\times\tau_{r}\left\{\displaystyle{\int_{\Gamma_{r}}G(\kappa,|\cdot-y_{r}|)(\nabla_{\Gamma_{r}}\tau_{r}^{-1}p)(y_{r})d\sigma(y_{r})}\right.\right.\vspace{2mm}\\&\hspace{5.3cm}\left.\left.+\displaystyle{\int_{\Gamma_{r}}G(\kappa,|\cdot-y_{r}|)(\mathbf{rot}_{\Gamma_{r}}\tau_{r}^{-1}q)(y_{r})d\sigma(y_{r})}\right\}(x)\right)\vspace{2mm}\\&+\kappa^{-1}\tau_{r}\left(\displaystyle{\int_{\Gamma_{r}}G(\kappa,|\cdot-y_{r}|)(\Delta_{\Gamma_{r}}\tau_{r}^{-1}p)(y_{r})d\sigma(y_{r})}\right)(x).\end{array}\end{equation*}

$\rhd$\textbf{ Integral representation of $\Pp_{r}M^{r}_{\kappa}\Pp_{r}^{-1}$ .}\newline 
 Recall that for all $\jj_{r}\in\TT\HH^{-\frac{1}{2}}(\Div_{\Gamma_{r}},\Gamma_{r})$, the operator $M_{\kappa}^r$ is defined by
\begin{equation*}\begin{array}{ll}M^{r}_{\kappa}\jj_{r}(x_{r})=&\displaystyle{\int_{\Gamma_{r}}\left((\nabla^{x_{r}}G(\kappa,|x_{r}-y_{r}|))\cdot\nn_{r}(x_{r})\right)\jj_{r}(y_{r})d\sigma(y_{r})}\vspace{2mm}\\&-\displaystyle{\int_{\Gamma_{r}}\nabla^{x_{r}}G(\kappa,|x_{r}-y_{r}|)\left(\nn_{r}(x_{r})\cdot\jj_{r}(y_{r})\right) d\sigma(y_{r})} .\vspace{2mm}\end{array}\end{equation*}
Using the equalities \eqref{eqd3} and the identity $\Rot\Rot=-\Delta+\nabla\Div$, we have
$$\begin{array}{rl} \Div_{\Gamma_{r}}M^r_{\kappa}\jj_{r}(x_{r}) 
=&\nn_{r}(x_{r})\cdot\displaystyle{\int_{\Gamma_{r}}\rot\rot^{x_{r}} \left(G(\kappa,|x_{r}-y_{r}|)\jj_{r}(y_{r}) \right) d\sigma(y_{r})}\\  =&\kappa^2\nn_{r}(x_{r})\cdot\displaystyle{\int_{\Gamma_{r}}\left(G(\kappa,|x_{r}-y_{r}|)\jj_{r}(y_{r}) \right)d\sigma(y_{r})} \vspace{2mm}\\ & +
\dfrac{\partial}{\partial\nn_{r}}\displaystyle{\int_{\Gamma_{r}} \left(G(\kappa,|x_{r}-y_{r}|)\Div_{\Gamma_{r}}\jj_{r}(y_{r}) \right)d\sigma(y_{r})}
\end{array}$$
 Proceeding by the same way than with the operator $\Pp_{r}C^{r}_{\kappa}\Pp_{r}^{-1}$, we obtain that the operator $\Pp_{r}M^{r}_{\kappa}\Pp_{r}^{-1}$ is defined for  $\jj=\nabla_{\Gamma}\;p+\mathbf{\Rot}_{\Gamma}\;q\in \TT\HH^{-\frac{1}{2}}(\Div_{\Gamma},\Gamma)$ by: 
 $$\Pp_{r}M^{r}_{\kappa}\Pp_{r}^{-1}\jj=\nabla_{\Gamma}P'(r)+\mathbf{rot}_{\Gamma}Q'(r),$$ with

\begin{equation*}\begin{array}{ll}P'(r)(x)=&\left(\tau_{r}\Delta_{\Gamma_{r}}^{-1}\tau_{r}^{-1}\right)\tau_{r}\left\{\kappa^2\displaystyle{\int_{\Gamma_{r}}\nn_{r}(\,\cdot\,)\cdot \left\{G(\kappa,|\cdot-y_{r}|)\Rot_{\Gamma_{r}}\tau_{r}^{-1}q(y_{r})\right\}d\sigma(y_{r})} \right.\vspace{2mm}\\&\qquad\qquad\quad\qquad+\,\kappa^{2}\displaystyle{\int_{\Gamma_{r}}\nn_{r}(\,\cdot\,)\cdot \left\{G(\kappa,|\cdot-y_{r}|)\nabla_{\Gamma_{r}}\tau_{r}^{-1}p(y_{r})\right\}d\sigma(y_{r})}\vspace{2mm}\\&\qquad\qquad\qquad+\left.\displaystyle{\int_{\Gamma_{r}}\dfrac{\partial}{\partial\nn_{r}(\,\cdot\,)}G(\kappa,|\cdot-y_{r}|)(\Delta_{\Gamma_{r}}\tau_{r}^{-1}p)(y_{r})d\sigma(y_{r})}\right\}(x),\end{array}\end{equation*}and

\begin{equation*}\begin{array}{ll}Q_{r}'(x)=&\left(\tau_{r}\Delta_{\Gamma_{r}}^{-1}\rot_{\Gamma_{r}}\tau_{r}^{-1}\right)\tau_{r}\left\{\displaystyle{\int_{\Gamma_{r}}\left(\nabla G(\kappa,|\cdot-y_{r}|)\cdot\nn_{r}(\,\cdot\,)\right)(\mathbf{rot}_{\Gamma_{r}}\tau_{r}^{-1}q)(y_{r})d\sigma(y_{r})} \right.\vspace{2mm}\\&\qquad\qquad\qquad+\displaystyle{\int_{\Gamma_{r}}(\left( \nabla G(\kappa,|\cdot-y_{r}|)\cdot\nn_{r}(\,\cdot\,)\right)(\nabla_{\Gamma_{r}}\tau_{r}^{-1} p)(y_{r})d\sigma(y_{r})}\vspace{2mm}\\&\qquad\qquad- \displaystyle{\int_{\Gamma_{r}}\nabla G(\kappa,|\cdot-y_{r}|)\left(\nn_{r}(\,\cdot\,)\cdot(\mathbf{rot}_{\Gamma_{r}}\tau_{r}^{-1} q)(y_{r})\right) d\sigma(y_{r})}\vspace{2mm}\\&\quad\quad-\left.\displaystyle{\int_{\Gamma_{r}}\nabla G(\kappa,|\cdot-y_{r}|)\left(\nn_{r}(\,\cdot\,)\cdot(\nabla_{\Gamma_{r}}\tau_{r}^{-1}p)(y_{r})\right) d\sigma(y_{r})}\right\}(x).\end{array}\end{equation*}
 These operators are composed of boundary integral operators with weakly singular and pseudo-homogeneous kernels of class -1 and of the surface differential operators defined in section \ref{BoundIntOp}. By  a change of variables in the integral, we then have to study the differentiability properties of the applications    $$\begin{array}{lcl}r&\mapsto&\tau_{r}\nabla_{\Gamma_{r}}\tau_{r}^{-1}\\r&\mapsto&\tau_{r}\Rot_{\Gamma_{r}}\tau_{r}^{-1}\\r&\mapsto&\tau_{r}\Div_{\Gamma_{r}}\tau_{r}^{-1}\\r&\mapsto&\tau_{r}\rot_{\Gamma_{r}}\tau_{r}^{-1}\\r&\mapsto&\tau_{r}\Delta_{\Gamma_{r}}\tau_{r}^{-1}\end{array}$$

%++++++++++++++++++++++++++++++++++++++++++++++++++++++
\subsection{G\^ateaux differentiability of the surface differential operators}
\begin{lemma}\label{nabla}The application $$\begin{array}{cccc}G:&B^{\infty}_{\epsilon}&\rightarrow&\mathscr{L}_{c}(H^{s+1}(\Gamma),\HH^{s}(\Gamma))\\&r&\mapsto&\tau_{r}\nabla_{\Gamma_{r}}\tau_{r}^{-1}\end{array}$$ is $\mathscr{C}^{\infty}$-G\^ateaux differentiable and its first derivative  at $r_{0}$ is  defined for $\xi\in\mathscr{C}^{\infty}_{b}(\R^3,\R^3)$ by
$$\frac{\partial G}{\partial r}[r_{0},\xi]u=-[G(r_{0})\xi]G(r_{0})u+\left(G(r_{0})u\cdot[G(r_{0})\xi]N(r_{0})\right)N(r_{0}).$$\end{lemma}
\begin{remark}Note that we can write $\dfrac{\partial N}{\partial r}[r_{0},\xi]=-[G(r_{0})\xi]N(r_{0})$. Since the first derivative of $N$ and $G$ can be expressed in function of $N$ and $G$ we obtain the G\^ateaux derivative of all order iteratively. 
\end{remark}

\begin{proof}In accordance to the definition \eqref{G} and the lemma \ref{N}, to prove the $\mathscr{C}^{\infty}$-G\^ateaux differentiability of $G$ we have to prove the $\mathscr{C}^{\infty}$- G\^ateaux differentiability of the application $$f:r\in B^{\infty}_{\epsilon}\mapsto\left\{u\mapsto \tau_{r}\left(\nabla\widetilde{\tau_{r}^{-1}u}\right)_{|_{\Gamma_{r}}}\right\}\in \mathscr{L}_{c}(H^{s+1}(\Gamma),\HH^{s}(\Gamma)).$$
 Let $x\in\Gamma$, we have $$\tau_{r}\left(\nabla\widetilde{\tau_{r}^{-1}u}\right)_{|_{\Gamma_{r}}}(x)=\nabla\left(\widetilde{u}\circ(\Id+r)^{-1}\right)_{|_{\Gamma_{r}}}(x+r(x))=\transposee{\left(\Id+\D r\right)_{|_{\Gamma_{r}}}^{-1}}(x+r(x))\circ\nabla\widetilde{u}_{|_{\Gamma}}(x),$$
and
$$\left(\Id+\D r\right)_{|_{\Gamma_{r}}}^{-1}(x+r(x))=\left[(\Id+\D r)_{|_{\Gamma}}(x)\right]^{-1}.$$The application $g:r\in B^{\infty}_{\epsilon}\mapsto(\Id+\D r)_{|_{\Gamma}}\in\mathscr{C}^{\infty}(\Gamma)$ is continuous, and $\mathscr{C}^{\infty}$-G\^ateaux differentiable. Its first derivative is  $\dfrac{\partial}{\partial r}g[0,\xi]=[\D\xi]_{|_{\Gamma}}$ and its higher order derivatives vanish. One can show that the application $h:r\in B^{\infty}_{\epsilon}\mapsto \left\{x\mapsto [g(r)(x)]^{-1}\right\}\in\mathscr{C}^{\infty}(\Gamma)$ is also $\mathscr{C}^{\infty}$ G\^ateaux-differentiable and that we have at $r_{0}$ and in the direction $\xi$:
 $$\dfrac{\partial h}{\partial r}[r_{0},\xi]=-h(r_{0})\circ\frac{\partial g}{\partial r}[r_{0},\xi]\circ h(r_{0})=-h(r_{0})\circ[\D\xi]_{|_{\Gamma}}\circ h(r_{0}).$$
 and 
 $$\dfrac{\partial^n h}{\partial r^n}[r_{0},\xi_{1},\hdots,\xi_{n}]=(-1)^n\sum_{\sigma\in\mathscr{S}_{n}}(\Id+\D r_{0})^{-1}\circ[\tau_{r_{0}}\D\tau_{r_{0}}^{-1}\xi_{\sigma(1)}]\circ\hdots\circ[\tau_{r_{0}}\D\tau_{r_{0}}^{-1}\xi_{\sigma(n)}]$$ where $\mathscr{S}_{n}$ is the permutation groupe of $\{1,\hdots,n\}$.
Finally we obtain the $\mathscr{C}^{\infty}-$ G\^ateaux differentiability of $f$ and we have 
$$\frac{\partial f}{\partial r}[r_{0},\xi]u=-[f(r_{0})\xi]f(r_{0})u.$$
To obtain the expression of the first derivative of $G$ we have to derive the following expression:
$$\begin{array}{cl}G(r)u&=(\tau_{r}\nabla_{\Gamma_{r}}\tau_{r}^{-1}u)=\tau_{r}\nabla \left(\widetilde{\tau_{r}^{-1}u}\right)-\left(\tau_{r}\nn_{r}\cdot\left(\tau_{r}\nabla\left( \widetilde{\tau_{r}^{-1}u}\right)\right)\right)\tau_{r}\nn_{r}\\&=f(r_{0})u-\left(f(r_{0})u\cdot N(r_{0})\right)N(r_{0}).\end{array}$$ By lemma \ref{N} and the chain and product rules we have
 $$\begin{array}{cl}\dfrac{\partial G}{\partial r}[r_{0},\xi]=&-[f(r_{0})\xi]f(r_{0})u+\left([f(r_{0})\xi]f(r_{0})u\cdot N(r_{0})\right)N(r_{0})\\&+\left(f(r_{0})u\cdot [G(r_{0})\xi]N(r_{0})\right)N(r_{0})+\left(f(r_{0})u\cdot N(r_{0})\right)[G(r_{0})\xi]N(r_{0})\end{array}$$
We had the first two terms in the right handside, it gives :
$$\begin{array}{rl}\dfrac{\partial G}{\partial r}[r_{0},\xi]=&-[G(r_{0})\xi]f(r_{0})u+\left(f(r_{0})u\cdot N(r_{0})\right)[G(r_{0})\xi]N(r_{0})\\&+\left(f(r_{0})u\cdot [G(r_{0})\xi]N(r_{0})\right)N(r_{0})\vspace{2mm}\\=&-[G(r_{0})\xi]G(r_{0})u+\left(f(r_{0})u\cdot [G(r_{0})\xi]N(r_{0})\right)N(r_{0}).\end{array}$$
 To conclude it suffice to note that $\left(f(r_{0})u\cdot [G(r_{0})\xi]N(r_{0})\right)=\left(G(r_{0})u\cdot [G(r_{0})\xi]N(r_{0})\right)$.
 \end{proof}
 
\begin{lemma}\label{D}The application $$\begin{array}{cccc}D:&B^{\infty}_{\epsilon}&\rightarrow&\mathscr{L}_{c}(\HH^{s+1}(\Gamma,\R^3),H^s(\Gamma))\\&r&\mapsto&\tau_{r}\Div_{\Gamma_{r}}\tau_{r}^{-1}\end{array}$$ is $\mathscr{C}^{\infty}$-G\^ateaux differentiable  and its first derivative  at $r_{0}$ is  defined for $\xi\in\mathscr{C}^{\infty}_{b}(\R^3,\R^3)$ by
$$\frac{\partial D}{\partial r}[r_{0},\xi]\uu=-\Tr([G(r_{0})\xi][G(r_{0})\uu])+\left([G(r_{0})\uu]N(r_{0})\cdot[G(r_{0})\xi]N(r_{0})\right).$$\end{lemma}

\begin{proof}For $\uu\in \HH^{s+1}(\Gamma,\R^n)$ we have $D(r)\uu=\Tr(G(r)\uu)$. Then we use the differentiation rules.
\begin{remark}Since the first derivative of $D$ is composed of $G$ and $N$ and the first  derivative of $J$ is composed of $J$ and $D$, we can obtain an expression of higher order derivatives of the jacobian iteratively. 
\end{remark}
\end{proof}

\begin{lemma}\label{rot2}The application $$\begin{array}{cccc}\mathbf{R}:& B^{\infty}_{\epsilon}&\rightarrow&\mathscr{L}_{c}(H^{s+1}(\Gamma),\HH^s(\Gamma))\\&r&\mapsto&\tau_{r}\Rot_{\Gamma_{r}}\tau_{r}^{-1}\end{array}$$ is $\mathscr{C}^{\infty}$-G\^ateaux differentiable and its first derivative at $r_{0}$ is defined for $\xi\in\mathscr{C}^{\infty}_{b}(\R^3,\R^3)$ by 
$$\frac{\partial \mathbf{R}}{\partial r}[r_{0},\xi]u=\transposee{[G(r_{0})\xi]}\mathbf{R}(r_{0})u-D(r_{0})\xi\cdot\mathbf{R}(r_{0})u.$$\end{lemma}

\begin{proof} Let  $u\in H^{s+1}(\Gamma)$. By definition, we have $\mathbf{R}(r_{0})u=G(r_{0})u\times N(r_{0})$. By lemmas \ref{N} and \ref{nabla} this application is $\mathscr{C}^{\infty}$ G\^ateaux differentiable. We have in $r_{0}$ and in the direction $\xi\in\mathscr{C}^{\infty}(\Gamma,\R^3)$
$$\frac{\partial \mathbf{R}}{\partial r}[r_{0},\xi]u=-\transposee{[G(r_{0})\xi]}G(r_{0})u\times N(r_{0})-G(r_{0})u\times[G(r_{0})\xi]N(r_{0}).$$
NB: recall that given a $(3\times3)$ matrix $A$ and vectors $b$ and $c$ we have $$Ab\times c+b\times Ac=\Tr(A)(b\times c)-\transposee{A}(b\times c).$$ We deduce the expression of the first derivatives  with $A=-[G(r_{0})\xi]$, $b=G(r_{0})u$ et $c=N(r_{0})$.
\end{proof}

\begin{lemma}\label{rot1}The application $$\begin{array}{cccc}R:&B^{\infty}_{\epsilon}&\rightarrow&\mathscr{L}_{c}(\HH^{s+1}(\Gamma),H^s(\Gamma))\\&r&\mapsto&\tau_{r}\rot_{\Gamma_{r}}\tau_{r}^{-1}\end{array}$$ is $\mathscr{C}^{\infty}$-G\^ateaux differentiable and its first derivative at $r_{0}$ is defined for $\xi\in\mathscr{C}^{\infty}_{b}(\R^3,\R^3)$ by 
$$\frac{\partial R}{\partial r}[r_{0},\xi]\uu=-\sum_{i=1}^3\left(G(r_{0})\xi^{i}\cdot\mathbf{R}(r_{0})u_{i}\right)- D(r_{0})\xi\cdot R(r_{0})\uu$$ where $\uu=(u_{1},u_{2},u_{3})$ and $\xi=(\xi^{1},\xi^{2},\xi^{3})$.\end{lemma}

\begin{proof} Let $\uu\in\HH^{s+1}(\Gamma,\R^3)$. By definition of the surface rotational we have $$R(r_{0})\uu=-\Tr(\mathbf{R}(r_{0})\uu).$$ We deduce the $\mathscr{C}^{\infty}$ differentiability of $R$ and the first derivative in $r_{0}$ in the direction $\xi$ is 
\begin{equation*}\begin{split}\frac{\partial R}{\partial r}[r_{0},\xi]\uu=&-\Tr\left(\frac{\partial \mathbf{R}}{\partial r}[r_{0},\xi]\uu\right)\\=&-\Tr\left(\transposee{[G(r_{0})\xi]}[\mathbf{R}(r_{0})\uu]\right)-D(r_{0})\xi\cdot \Tr\left(-\mathbf{R}(r_{0})\uu]\right)\\=&-\sum_{i=1}^3\left(G(r_{0})\xi_{i}\cdot\mathbf{R}(r_{0})\uu_{i}\right)- D(r_{0})\xi\cdot R(r_{0})\uu.\end{split}\end{equation*}
\end{proof}Here again we can obtain higher order derivatives of these operators iteratively.
\begin{remark} One can see that  we do not need more than the first derivative of the deformations $\xi$.  Thus these results hold true for boundaries and deformations of class $\mathscr{C}^{k+1}$, $k\in\N^*$ with  differential operators  considered in $\mathscr{L}_{c}(\mathscr{C}^{k+1}(\Gamma),\mathscr{C}^k(\Gamma))$.\end{remark}
When $\xi=\nn$ we obtain the commutators
\begin{equation}\label{commutators}\begin{split}\frac{\partial}{\partial\nn}(\nabla_{\Gamma}u)-\nabla_{\Gamma}\left(\frac{\partial}{\partial\nn}u\right)=&\;-\mathcal{R}_{\Gamma}\nabla_{\Gamma}u\\\frac{\partial}{\partial\nn}(\Rot_{\Gamma}u)-\Rot_{\Gamma}\left(\frac{\partial}{\partial\nn}u\right)=&\;\mathcal{R}_{\Gamma}\Rot_{\Gamma}u-\mathcal{H}_{\Gamma}\Rot_{\Gamma}u\\\frac{\partial}{\partial\nn}(\Div_{\Gamma}\uu)-\Div_{\Gamma}\left(\frac{\partial}{\partial\nn}\uu\right)=&\;-\Tr(\mathcal{R}_{\Gamma}[\nabla_{\Gamma}\uu])\\\frac{\partial}{\partial\nn}(\rot_{\Gamma}\uu)-\rot_{\Gamma}\left(\frac{\partial}{\partial\nn}\uu\right)=&\;-\Tr(\mathcal{R}_{\Gamma}[\Rot_{\Gamma}\uu])-\mathcal{H}_{\Gamma}\rot_{\Gamma}\uu\end{split}\end{equation}
where $R_{\Gamma}=[\nabla_{\Gamma}\nn]$ and $\mathcal{H}_{\Gamma}=\Div_{\Gamma}\nn$.

From the precedent results we have:

\begin{lemma}\label{delta}The application $$\begin{array}{cccc}L :& B^{\infty}_{\epsilon}&\rightarrow&\mathscr{L}_{c}(H^{s+2}(\Gamma),H^{s}(\Gamma))\\&r&\mapsto&\tau_{r}\Delta_{\Gamma_{r}}\tau_{r}^{-1}\end{array}$$ is $\mathscr{C}^{\infty}$-G\^ateaux differentiable. \end{lemma}

\begin{proof} It suffice to write:
$$\tau_{r}\Delta_{\Gamma_{r}}\tau_{r}^{-1}=(\tau_{r}\Div_{\Gamma_{r}}\tau_{r}^{-1})(\tau_{r}\nabla_{\Gamma_{r}}\tau_{r}^{-1})=-(\tau_{r}\rot_{\Gamma_{r}}\tau_{r}^{-1})(\tau_{r}\Rot_{\Gamma_{r}}\tau_{r}^{-1}).$$ The operators $\tau_{r}\Delta_{\Gamma_{r}}\tau_{r}^{-1}$ is composed of operators infinitely G\^ateaux differentiable.\end{proof}

 View the integral representations of the  operators $\Pp_{r}C_{\kappa}^r\Pp_{r}^{-1}$ and $\Pp_{r}M_{\kappa}^r\Pp_{r}^{-1}$ we have to  study the G\^ateaux differentiability of the  applications $r\mapsto\tau_{r}\Delta_{\Gamma_{r}}^{-1}\rot_{\Gamma_{r}}\tau_{r}^{-1}$  and $r\mapsto\tau_{r}\Delta_{\Gamma_{r}}^{-1}\rot_{\Gamma_{r}}\tau_{r}^{-1}$. We have seen that for $r\in B^{\infty}_{\epsilon}$ the operator $\rot_{\Gamma_{r}}$ is linear and continuous from $\HH^{s+1}(\Gamma_{r})$ to $H^s_{*}(\Gamma_{r})$, that the operator  $\Div_{\Gamma_{r}}$ is linear and continuous from $\TT\HH^{s+1}(\Gamma_{r})$ in  $H_{*}^s(\Gamma_{r})$ and that $\Delta_{\Gamma_{r}}^{-1}$ is defined from $H^{s}_{*}(\Gamma_{r})$ in $H^{s+2}(\Gamma_{r})\slash\R$. To use the chain rules,  it is  important to construct derivatives in $r=0$ between  the spaces $\HH^{s+1}(\Gamma)$ and $H^s_{*}(\Gamma)$ for the scalar curl operator, between the spaces $\TT\HH^{s+1}(\Gamma)$ and $H^s_{*}(\Gamma)$  for the divergence operator and between the spaces   $H^s_{*}(\Gamma)$ and $H^{s+2}(\Gamma)/\R$ for the Laplace-Beltrami operator. As an alternative we use the :

 \begin{proposition} Let $u$ be a scalar  function defined on $\Gamma_{r}$. Then $u_{r}\in H^s_{*}(\Gamma_{r})$ if and only if $J_{r}\tau_{r}u_{r}=J_{r}u_{r}\circ(\Id+r)\in H^s_{*}(\Gamma)$.
\end{proposition}

 As a consequence the applications $r\mapsto J_{r}\tau_{r}\rot_{\Gamma_{r}}\tau_{r}^{-1}$ and $r\mapsto   J_{r}\tau_{r}\Div_{\Gamma_{r}}\pi^{-1}(r)$ are well-defined from $\HH^{s+1}(\Gamma)$ and $\TT\HH^{s+1}(\Gamma)$ respectively to $H^s_{*}(\Gamma)$.

\begin{lemma}The applications $$\begin{array}{ccc}B^{\infty}_{\epsilon}&\rightarrow&\mathscr{L}_{c}(\HH^{s+1}(\Gamma),\HH^{s}_{*}(\Gamma))\\r&\mapsto&J_{r}\tau_{r}\rot_{\Gamma_{r}}\tau_{r}^{-1}\end{array}\text{ and }\begin{array}{ccc}B^{\infty}_{\epsilon}&\rightarrow&\mathscr{L}_{c}(\TT\HH^{s+1}(\Gamma),\HH^{s}_{*}(\Gamma))\\r&\mapsto&J_{r}\tau_{r}\Div_{\Gamma_{r}}\pi^{-1}(r)\end{array}$$ are $\mathscr{C}^{\infty}$-G\^ateaux differentiable end their first derivatives at $r=0$ defined for $\xi\in\mathscr{C}^{\infty}_{b}(\R^3,\R^3)$ by 
$$\frac{\partial J_{r}\tau_{r}\rot_{\Gamma_{r}}\tau_{r}^{-1}}{\partial r}[0,\xi]\uu=-\sum_{i=1}^3\nabla_{\Gamma}\xi_{i}\cdot\Rot_{\Gamma}\uu_{i}.$$
and $$\begin{array}{ll}\dfrac{\partial J_{r}\tau_{r}\Div_{\Gamma_{r}}\pi^{-1}(r)}{\partial r}[0,\xi]\uu=&-\Tr([\nabla_{\Gamma}\xi]\nabla_{\Gamma}\uu)+\Div_{\Gamma}\xi\Div_{\Gamma}\uu +\left([\nabla_{\Gamma}\uu]\nn\cdot[\nabla_{\Gamma}\xi]\nn\right)\\&+\left(\uu\cdot[\nabla_{\Gamma}\xi]\nn\right)\mathcal{H}_{\Gamma}.\end{array}$$\end{lemma}

\begin{proof} Let $\uu\in\TT\HH^{s+1}(\Gamma)$
We have $$\frac{\partial \tau_{r}\pi^{-1}(r)}{\partial r}[0,\xi]\uu= \left(\uu\cdot[\nabla_{\Gamma}\xi]\nn\right)\nn.$$
Next we use the lemma \ref{J}, \ref{D} and \ref{rot1}.
 For $u\in\HH^s(\Gamma)$, it is clear that $\sum_{i=1}^3\nabla_{\Gamma}\xi_{i}\cdot\Rot_{\Gamma}\uu_{i}$ is of vanishing mean value since the space $\nabla_{\Gamma}H^s(\Gamma)$ is orthogonal to $\Rot_{\Gamma}H^s(\Gamma)$ for the  $\LL^2$ scalar product. An other  argument whithout using the explicit form of the derivatives is : for all  $\uu\in\HH^{s+1}(\Gamma)$, we derive the application $$r\mapsto\int_{\Gamma}J_{r}\tau_{r}\rot_{\Gamma_{r}}\tau_{r}^{-1}\uu \,d\sigma\equiv0.$$ and for $\uu\in\TT\HH^{s+1}(\Gamma)$ we derive the application $$r\mapsto\int_{\Gamma}J_{r}\tau_{r}\Div_{\Gamma_{r}}\pi^{-1}(r)\uu \,d\sigma\equiv0.$$

\end{proof}
%\begin{lemma}\label{div}The application $$\begin{array}{ccc}B_{\delta}&\rightarrow&\mathscr{L}_{c}(\HH_{\times}^s(\Gamma),\HH^{s-1}_{*}(\Gamma))\\r&\mapsto&J_{r}\tau_{r}\Div_{\Gamma_{r}}\pi(r)\end{array}$$ is $\mathscr{C}$-G\^ateaux diff\'erentiable and we have
%$$\frac{\partial J_{r}\tau_{r}\Div_{\Gamma_{r}}\pi(r)}{\partial r}[0,\xi]u=-\sum_{i=1}^3\nabla_{\Gamma}\xi_{i}\cdot\Rot_{\Gamma}(\uu\times\nn)_{i}+\rot_{\Gamma}\left((\uu\times[\nabla_{\Gamma}\xi]\nn\right).$$\end{lemma}
%\begin{proof}Pour $\uu\in \HH^s(\Gamma)$ on a $\Div_{\Gamma_{r}}\pi(r)\uu=\rot_{\Gamma_{r}}(\pi(r)\uu\times\nn_{r})$.
%\end{proof}
  Let us note that $u_{r}\in H^s(\Gamma_{r})/\R$ if and only if  $\tau_{r}u_{r}\in H^s(\Gamma)/\R$.
\begin{lemma}\label{delta-1} The application $$\begin{array}{ccc}B_{\epsilon}^{\infty}&\rightarrow&\mathscr{L}_{c}(H^{s+2}_{*}(\Gamma), H^s(\Gamma)\slash\R)\\r&\mapsto&\tau_{r}\Delta_{\Gamma_{r}}^{-1}\tau_{r}^{-1}J_{r}^{-1}\end{array}$$ is $\mathscr{C}^{\infty}$-G\^ateaux differentiable.\end{lemma}

\begin{proof}We have seen in section \ref{BoundIntOp}, that the Laplace-Beltrami operator is invertible  from $H^{s+2}(\Gamma_{r})\slash\R$ to $H^s_{*}(\Gamma_{r})$.  As a consequence $J_{r}\tau_{r}\Delta_{\Gamma_{r}}\tau_{r}^{-1}$ is invertible from $H^{s+2}(\Gamma)\slash\R$ to $H^s_{*}(\Gamma)$. By lemma \ref{d-1} we deduce that $r\mapsto\tau_{r}\Delta_{\Gamma_{r}}^{-1}\tau_{r}^{-1}J_{r}^{-1}$ is $\mathscr{C}^{\infty}$-G\^ateaux  differentiable and that we have
$$\frac{\partial\;\tau_{r}\Delta_{\Gamma_{r}}^{-1}\tau_{r}^{-1}J_{r}^{-1}}{\partial r}[0,\xi]=\;-\Delta_{\Gamma}^{-1}\circ\left(\frac{\partial\;J_{r}\tau_{r}\Delta_{\Gamma_{r}}\tau_{r}^{-1}}{\partial r}[0,\xi]\right)\circ\Delta_{\Gamma}^{-1}.$$   \end{proof}

Now we have all the tools to establish the differentiability properties of the electromagnetic boundary integral operators and then of the solution to the dielectric scattering problem.

%++++++++++++++++++++++++++++++++++++++++++++++++++++++++++++++
 \subsection{Shape derivatives of the solution to the  dielectric  problem}
  For more simplicity in the  writing we use the following notations :
$$\Psi_{E_{\kappa}}(r)=\Psi_{E_{\kappa}}^r\Pp_{r}^{-1},\;\Psi_{M_{\kappa}}(r)=\Psi_{M_{\kappa}}^r\Pp_{r}^{-1},\;C_{\kappa}(r)=\Pp_{r}C^r_{{\kappa}}\Pp_{r}^{-1},\text{ et }M_{\kappa}(r)=\Pp_{r}M^r_{{\kappa}}\Pp_{r}^{-1}.$$
\begin{theorem}\label{thpsi} The applications  $$\begin{array}{rcl}B_{\epsilon}^{\infty}&\rightarrow& \mathscr{L}_{c}(\TT\HH^{-\frac{1}{2}}(\Div_{\Gamma},\Gamma), \HH(\Rot,K_{p}))\\r&\mapsto&\Psi_{E_{\kappa}}(r)\\r&\mapsto&\Psi_{M_{\kappa}}(r)\end{array}$$ are infinitely G\^ateaux differentiable.
Moreover, their first derivative at $r=0$ can be extended in linear an bounded operators from $\TT\HH^{\frac{1}{2}}(\Div_{\Gamma},\Gamma)$ to $\HH(\Rot,\Omega)$ and $\HH_{loc}(\Rot,\Omega^c)$ and given  $\jj\in\TT\HH^{\frac{1}{2}}(\Div_{\Gamma},\Gamma)$ the potentials $\dfrac{\partial\Psi_{E_{\kappa}}}{\partial r}[0,\xi]\jj$ et $\dfrac{\partial\Psi_{M_{\kappa}}}{\partial r}[0,\xi]\jj$ satisfy the Maxwell's equations 
$$\Rot\Rot\uu-\kappa^2\uu=0$$
in $\Omega$ and $\Omega^c$ and the Silver-M\"uller condition. 
\end{theorem}

\begin{proof} Let $\jj\in\TT\HH^{-\frac{1}{2}}(\Div_{\Gamma},\Gamma)$ and $\nabla_{\Gamma}p+\Rot_{\Gamma}q$ its Helmholtz decomposition. Recall that $\Psi_{E_{\kappa}}(r)\jj$ and $\Psi_{M_{\kappa}}(r)\jj$ can be written: 
\begin{equation*}\begin{split}\Psi_{E_{\kappa}}(r)\jj=&\kappa \Psi^{r}_{\kappa}\tau_{r}^{-1}(\tau_{r}\mathbf{P_{r}}^{-1}\jj)-\kappa^{-1}\nabla \Psi^{r}_{\kappa}\tau_{r}^{-1}(\tau_{r}\Delta_{\Gamma_{r}}\tau_{r}^{-1}p),\vspace{2mm} \\\Psi_{E_{\kappa}}(r)\jj=&\rot\psi_{\kappa}^{r}\tau_{r}^{-1}(\tau_{r}\mathbf{P_{r}}^{-1}\jj).\end{split}\end{equation*}

By composition of differentiable applications, we deduce that $r\mapsto\Psi_{E_{\kappa}}(r)$ and $r\mapsto~\Psi_{M_{\kappa}}(r)$ are infinitely G\^ateaux differentiable far from the boundary and that their first derivatives are continuous from $\TT\HH^{\frac{1}{2}}(\Div_{\Gamma},\Gamma)$ to $\LL^2(\Omega)\cup \LL^2_{loc}(\Omega^c)$. Recall that we have,
$$\rot\Psi_{E_{\kappa}}(r)\jj=\kappa\Psi_{M_{\kappa}}(r)\jj\text{ and }\rot\Psi_{M_{\kappa}}(r)\jj=\kappa\Psi_{E_{\kappa}}(r)\jj.$$ Far from the boundary  we can invert the differentiation with respect to $x$ and the derivation with respect to $r$ and it gives: $$\Rot\frac{\partial \Psi_{E_{\kappa}}}{\partial r}[0,\xi]\jj=\kappa\frac{\partial\Psi_{M_{\kappa}}}{\partial r}[0,\xi]\jj\text{ et }\Rot\frac{\partial \Psi_{M_{\kappa}}}{\partial r}[0,\xi]\jj=\kappa\frac{\partial\Psi}{\partial r}[0,\xi]\jj.$$
It follows that $\dfrac{\partial \Psi_{E_{\kappa}}}{\partial r}[0,\xi]\jj$ and $\dfrac{\partial \Psi_{M_{\kappa}}}{\partial r}[0,\xi]\jj$ are in $\HH(\Rot,\Omega)\cup\HH_{loc}(\Rot,\Omega^c)$  and that  they satisfy the Maxwell equations and the Silver-M\"uller condition.
\end{proof}

 We recall that the operator $C_{\kappa}(r)$ admit the following representation :
\begin{equation}C_{\kappa}(r)=\label{C}\mathbf{P_{r}}C^{r}_{\kappa}\mathbf{P_{r}}^{-1}\jj=\nabla_{\Gamma}P(r)+\Rot_{\Gamma}Q(r),\end{equation} 
where
\begin{equation*}\label{P}\begin{array}{ll}P(r)=-\kappa\;(J_{r}\tau_{r}\Delta_{\Gamma_{r}}\tau_{r}^{-1})^{-1}(J_{r}\tau_{r}\rot_{\Gamma_{r}}\tau_{r}^{-1})\left(\tau_{r}V^{r}_{\kappa}\tau_{r}^{-1}\right)\left[\left(\tau_{r}\nabla_{\Gamma_{r}}\tau_{r}^{-1}p\right)+\left(\tau_{r}\Rot_{\Gamma_{r}} \tau_{r}^{-1}q\right)\right]\end{array}\end{equation*}and $Q(r)=$
\begin{equation*}\label{Q}\begin{array}{ll}\hspace{-2mm}&\hspace{-2mm}-\kappa\;(J_{r}\tau_{r}\Delta_{\Gamma_{r}}\tau_{r}^{-1})^{-1}(J_{r}\tau_{r}\Div_{\Gamma_{r}}\pi^{-1}(r))\pi(r)\left(\tau_{r}V^{r}_{\kappa}\tau_{r}^{-1}\right)\left[\left(\tau_{r}\nabla_{\Gamma_{r}}\tau_{r}^{-1}p\right)+\left(\tau_{r}\Rot_{\Gamma_{r}} \tau_{r}^{-1}q\right)\right]\vspace{3mm}\\&+\kappa^{-1}\left(\tau_{r}V^{r}_{\kappa}\tau_{r}^{-1}\right)\left(\tau_{r}\Delta_{\Gamma_{r}}\tau_{r}^{-1}p\right).\end{array}\end{equation*}

\begin{remark}Let $\jj\in \TT\HH^{s}(\Div_{\Gamma},\Gamma)$ and $\nabla_{\Gamma}\;p+\Rot_{\Gamma}\;q$ its helmholtz decomposition. We  want to derive: $$\begin{array}{ll}\Pp_{r}C^{r}_{\kappa}\Pp_{r}^{-1}\jj&=\Pp_{r}C^{r}_{\kappa}(\nabla_{\Gamma_{r}}\tau_{r}^{-1}p+\Rot_{\Gamma_{r}}\tau_{r}^{-1}q)\\&=\Pp_{r}(\nabla_{\Gamma_{r}}P_{r}+\Rot_{\Gamma_{r}}Q_{r})\\&=\nabla_{\Gamma}P(r)+\Rot_{\Gamma}Q(r).\end{array}$$ We have:$$\frac{\partial \mathbf{P_{r}}C^{r}_{\kappa}\mathbf{P_{r}}^{-1}\jj}{\partial r}[0,\xi]=\nabla_{\Gamma}\frac{\partial P}{\partial r}[0,\xi]+\mathbf{rot}_{\Gamma}\frac{\partial Q}{\partial r}[0,\xi].$$ The derivative with respect to $r$ \`a $r$ de $\mathbf{P_{r}}C^{r}_{\kappa}\mathbf{P_{r}}^{-1}\jj$ is given by the derivatives of the functions $P(r)=\tau_{r}(P_{r})$ and of $Q(r)=\tau_{r}(Q_{r})$.\newline
We also have $$\frac{\partial \; \pi(r)f(r)}{\partial r}[0,\xi]=\pi(0)\frac{\partial \; f(r)}{\partial r}[0,\xi]$$

\end{remark}
By composition of infinite differentiable applications we obtain the 
\begin{theorem} The application: 
$$\begin{array}{lcl}B_{\epsilon}^{\infty}&\rightarrow &\mathscr{L}_{c}\left(\TT\HH^{-\frac{1}{2}}(\Div_{\Gamma},\Gamma),\TT\HH^{-\frac{1}{2}}(\Div_{\Gamma},\Gamma)\right)\\r&\mapsto&\Pp_{r}C^{r}_{\kappa}\Pp_{r}^{-1}\end{array}$$ is infinitely G\^ateaux differentiable.\end{theorem}
Recall that the operator $\Pp_{r}M^{r}_{\kappa}\mathbf{P_{r}}^{-1}$ admit the following representation :
\begin{equation*}\label{M}\mathbf{P_{r}}M^{r}_{\kappa}\mathbf{P_{r}}^{-1}\jj=\nabla_{\Gamma}P'(r)+\Rot_{\Gamma}Q'(r),\end{equation*}where
\begin{equation*}\label{P'}\begin{array}{rl}P'(r)=&\left(J_{r}\tau_{r}\Delta_{\Gamma_{r}}\tau_{r}^{-1}\right)^{-1}(\kappa^2J_{r}\tau_{r}\nn_{r}\cdot (\tau_{r}V^r_{\kappa}\tau_{r}^{-1})\left[\left(\tau_{r}\nabla_{\Gamma_{r}}\tau_{r}^{-1}p\right)+\left(\tau_{r}\Rot_{\Gamma_{r}} \tau_{r}^{-1}q\right)\right]\vspace{2mm}\\& +\left(J_{r}\tau_{r}\Delta_{\Gamma_{r}}\tau_{r}^{-1}\right)^{-1}(J_{r}\tau_{r}D_{\kappa}^{r}\tau_{r}^{-1})(\tau_{r}\Delta_{\Gamma_{r}}\tau_{r}^{-1}p)\end{array}\end{equation*}and
 $Q'(r)=$
\begin{equation*}\label{Q'}\begin{array}{ll}\left(J_{r}\tau_{r}\Delta_{\Gamma_{r}}\tau_{r}^{-1}\right)^{-1}(J_{r}\tau_{r}\rot_{\Gamma_{r}}\tau_{r}^{-1})(\tau_{r}(B_{\kappa}^{r}-D_{\kappa}^{r})\tau_{r}^{-1})\left[\left(\tau_{r}\nabla_{\Gamma_{r}}\tau_{r}^{-1}p\right)+\left(\tau_{r}\Rot_{\Gamma_{r}} \tau_{r}^{-1}q\right)\right]\end{array}\end{equation*}
with
\begin{equation*}\begin{array}{ll}\tau_{r}B_{k}^r\Pp_{r}^{-1}\jj=&\tau_{r}\left\{\displaystyle{\int_{\Gamma_{r}}\nabla G(\kappa,|\cdot-y_{r}|)\left(\nn_{r}(\,\cdot\,)\cdot(\nabla_{\Gamma_{r}}\tau_{r}^{-1} p)(y_{r})\right)d\sigma(y_{r})}\right. \vspace{2mm}\\&\hspace{5mm}+\left.\displaystyle{\int_{\Gamma_{r}}\nabla G(\kappa,|\cdot-y_{r}|)\left(\nn_{r}(\,\cdot\,)\cdot(\Rot_{\Gamma_{r}}\tau_{r}^{-1} q)(y_{r})\right)d\sigma(y_{r})}\Big\}\right\}.\end{array}\end{equation*}
\newline

\begin{theorem}The application:
$$\begin{array}{lcl}B_{\epsilon}^{\infty}&\rightarrow& \mathscr{L}_{c}\left(\TT\HH^{-\frac{1}{2}}(\Div_{\Gamma},\Gamma),\TT\HH^{-\frac{1}{2}}(\Div_{\Gamma},\Gamma)\right)\\r&\mapsto&\Pp_{r}M_{\kappa}\Pp_{r}^{-1}\end{array}$$ is infinitely G\^ateaux differentiable and the G\^ateaux derivatives have the same regularity than $M_{\kappa}$ so that it is compact.\end{theorem}

\begin{proof} By composition of infinite differentiable applications it remains to prove the infinite G\^ateaux differentiability of the application 
$$\begin{array}{lcl}B_{\delta}&\rightarrow& \mathscr{L}_{c}\left(\TT\HH^{-\frac{1}{2}}(\Div_{\Gamma},\Gamma),\HH^{\frac{1}{2}}(\Gamma)\right)\\r&\mapsto&\tau_{r}B^r_{\kappa}\Pp_{r}^{-1}.\end{array}$$
The function $(x,y-x)\mapsto \nabla G(\kappa,|x-y|)$ is pseudo-homogeneous of class 0. We then have to prove that for any fixed $(x,y)\in(\Gamma\times\Gamma)^*$ and any function $p\in H^{\frac{3}{2}}(\Gamma)$ the G\^ateaux derivatives of
$$r\mapsto (\tau_{r}\nn_{r})(x)\cdot\left(\tau_{r}\nabla_{\Gamma_{r}}\tau_{r}^{-1}p\right)(y)$$ behave as $|x-y|^2$ when $x-y$ tends to zero. To do so, either we write 
$$(\tau_{r}\nn_{r})(x)\cdot\left(\tau_{r}\nabla_{\Gamma_{r}}\tau_{r}^{-1}p\right)(y)=\left((\tau_{r}\nn_{r})(x)-(\tau_{r}\nn_{r})(y)\right)\cdot\left(\tau_{r}\nabla_{\Gamma_{r}}\tau_{r}^{-1}p\right)(y)$$
or we use lemmas \ref{N} and \ref{G}.
\end{proof}

\begin{theorem} \label{shapederiv}Assume that :\newline
1) $\EE^{inc}\in\HH^1_{loc}(\Rot,\R^3)$ and \newline
2)  the applications
$$\begin{array}{rcl}B_{\epsilon}^{\infty}&\rightarrow&\TT\HH^{-\frac{1}{2}}(\Div_{\Gamma},\Gamma)\\r&\mapsto&\Pp_{r}\left(\nn_{r}\times\EE^{inc}_{|\Gamma_{r}}\right)\\r&\mapsto& \Pp_{r}\left(\nn_{r}\times\left(\Rot\EE^{inc}\right)_{|{\Gamma_{r}}}\right)\end{array}$$ are G\^ateaux differentiable at $r=0$. Then the application mapping $r$ onto the solution $\mathscr{E}(r)=\EE(\Omega_{r})\in\HH(\Rot,\Omega)\cup\HH_{loc}(\Rot,\Omega^c)$ to the scattering problem the obstacle $\Omega_{r}$ is G\^ateaux differentiable at $r=0$.
\end{theorem}
\begin{proof} By composition of differentiable applications.
We write for the exterior field $\mathscr{E}^s$:

 \begin{equation*}\begin{split}\frac{\partial \mathscr{E}^s}{\partial r}[0,\xi]=&\left(-\frac{\partial \Psi_{E_{\kappa_{e}}}}{\partial r}[0,\xi]-i\eta\frac{\partial\Psi_{M_{\kappa_{e}}}}{\partial r}[0,\xi]C_{0}-i\eta\Psi_{M_{\kappa_{e}}}\frac{\partial C_{0}}{\partial r}[0,\xi]\right)\jj\\
 &+(-\Psi_{E_{\kappa_{e}}}-i\eta\Psi_{M_{\kappa_{e}}}C^*_{0})\SS^{-1}\left(-\frac{\partial \SS}{\partial r}[0,\xi]\jj\right)\\&+(-\Psi_{E_{\kappa_{e}}}-i\eta\Psi_{M_{\kappa_{e}}}C^*_{0})\SS^{-1}\left(-\rho\frac{\partial M_{\kappa_{i}}}{\partial r}[0,\xi]\gamma_{D}\EE^{inc}-\frac{\partial C_{\kappa_{i}}}{\partial r}[0,\xi]\gamma_{N_{\kappa_{e}}}\EE^{inc}\right)\\&+(-\Psi_{E_{\kappa_{e}}}-i\eta\Psi_{M_{\kappa_{e}}}C^*_{0})\SS^{-1}\left(-\rho\left(\frac{1}{2}+M_{\kappa_{i}}\right)\frac{\partial \Pp_{r}\gamma_{D}^{r}\EE^{inc}}{\partial r}[0,\xi]\right)\\&+(-\Psi_{E_{\kappa_{e}}}-i\eta\Psi_{M_{\kappa_{e}}}C^*_{0})\SS^{-1}\left(-C_{\kappa_{i}}\frac{\partial \Pp_{r}\gamma_{N_{\kappa_{e}}}^{r}\EE^{inc}}{\partial r}[0,\xi]\right).\end{split}\end{equation*}
The condition 1) guarantees that the solution $\jj\in\TT\HH^{\frac{1}{2}}(\Div_{\Gamma},\Gamma)$ so that the first term in the right handside are in $\HH_{loc}(\Rot,\Omega^c)$ and the second condition guarantees that the last two term  is in $\HH_{loc}(\Rot,\Omega^c)$.
Of the same for the interior field we write: 
 \begin{equation*}\begin{split}\frac{\partial \mathscr{E}^{i}}{\partial r}[0,\xi]=&-\frac{1}{\rho}\frac{\partial \Psi_{E_{\kappa_{i}}}}{\partial r}[0,\xi]\gamma_{N_{\kappa_{e}}}^c\left(\EE^{s}+\EE^{inc}\right)-\frac{\partial\Psi_{M_{\kappa_{i}}}}{\partial r}[0,\xi]\gamma_{D}^c\left(\EE^{s}+\EE^{inc}\right)\\ &-\frac{1}{\rho}\Psi_{E_{\kappa_{i}}}\frac{\partial\Pp_{r}\gamma_{N_{\kappa_{e}}}^{r}\left(\mathscr{E}^s(r)+\EE^{inc}\right)}{\partial r}[0,\xi]-\Psi_{M_{\kappa_{i}}}\frac{\partial\Pp_{r}\gamma_{D}^{r}\left(\mathscr{E}^s(r)+\EE^{inc}\right)}{\partial r}[0,\xi]\end{split}
\end{equation*}
The condition 1) guarantees that $\gamma_{N_{\kappa_{e}}}^c\left(\EE^{s}+\EE^{inc}\right)$ and $\gamma_{D}^c\left(\EE^{s}+\EE^{inc}\right)$ are in $\TT\HH^{\frac{1}{2}}(\Div_{\Gamma},\Gamma)$ so that the first two terms are in $\HH(\Rot,\Omega)$ and the second condition guarantees that the last two term  is in $\HH(\Rot,\Omega)$.
\end{proof} 
 
 \begin{theorem}The application mapping $r$ to  the far field pattern $\EE^{\infty}(\Omega_{r})\in\TT\mathscr{C}^{\infty}(S^2)$ of the solution to the scattering problem the obstacle $\Omega_{r}$ is $\mathscr{C}^{\infty}$-G\^ateaux differentiable.
\end{theorem}
 
 %+++++++++++++++++++++++++++++++++++++
\subsection{Characterisation of the first derivative}

 The following theorem give a  caracterisation of the first G\^ateaux derivative  of $r\mapsto\mathscr{E}(r)$ in $r=0$.

\begin{theorem}Under the hypothesis of theorem \ref{shapederiv} the first derivative at $r=0$ in the direction $\xi$  solve the following scattering problem :
\begin{equation}\left\{\begin{split}\Rot\Rot\frac{\partial \mathscr{E}^{i}}{\partial r}[0,\xi]-\kappa_{i}^2\frac{\partial \mathscr{E}^{i}}{\partial r}[0,\xi]=0\\\Rot\Rot\frac{\partial \mathscr{E}^{s}}{\partial r}[0,\xi]-\kappa_{e}^2\frac{\partial \mathscr{E}^{s}}{\partial r}[0,\xi]=0\end{split}\right.\end{equation} with the boundary conditions :
\begin{equation}\left\{\begin{split}\nn\times\frac{\partial \mathscr{E}^{i}}{\partial r}[0,\xi]-\nn\times\frac{\partial \mathscr{E}^{s}}{\partial r}[0,\xi]=g_{D}\\\mu_{i}^{-1}\nn\times\Rot\frac{\partial \mathscr{E}^{i}}{\partial r}[0,\xi]-\mu_{e}^{-1}\nn\times\Rot\frac{\partial \mathscr{E}^{s}}{\partial r}[0,\xi]=g_{N},\end{split}\right.\end{equation}

where\begin{equation*}\begin{split}g_{D}=&-\left(\xi\cdot\nn\right)\nn\times\frac{\partial}{\partial\nn}\left(\EE^{i}-\EE^s-\EE^{inc}\right)\\&+\Rot_{\Gamma}(\xi\cdot\nn)\;\nn\cdot\left(\EE^{i}-\EE^s-\EE^{inc}\right),\end{split}\end{equation*} and 

\begin{equation*}\begin{split}g_{N}=&-\left(\xi\cdot\nn\right)\nn\times\frac{\partial}{\partial\nn}\left(\mu_{i}^{-1}\Rot\EE^{i}-\mu_{e}^{-1}\Rot\EE^s-\mu_{e}^{-1}\Rot\EE^{inc}\right)\\&+\Rot_{\Gamma}(\xi\cdot\nn)\;\nn\cdot\left(\mu_{i}^{-1}\Rot\EE^{i}-\mu_{e}^{-1}\Rot\EE^s-\mu_{e}^{-1}\Rot\EE^{inc}\right).\end{split}\end{equation*} and where $\dfrac{\partial\mathscr{E}^{s}}{\partial r}[0,\xi]$ satisfies the Silver-M\"uller condition.
\end{theorem}

\begin{proof}
 We have shown in the previous paragraph that the potential operators and their G\^ateaux derivatives satisfy the Maxwell' equations and the Silver-M\"uller condition. 
 It remains to compute the boundary conditions. We could use the integral representation as Potthast did but it would need to write  too long formula. 
  For $x\in\Gamma$ we derive in $r=0$ the expression:
\begin{equation}\label{bvder}\nn_{r}(x+r(x))\times\left(\mathscr{E}^i(r)(x+r(x))-\mathscr{E}^s(r)(x+r(x))-\EE^{inc}(x+r(x))\right)=0.\end{equation} It gives in the direction $\xi$:
$$\begin{array}{rl}0=&\dfrac{\partial\tau_{r}\nn_{r}}{\partial r}[0,\xi](x)\times\left(\EE^{i}(x)-\EE^s(x)-\EE^{inc}(x)\right)\\&+\nn(x)\times\left(\dfrac{\partial\mathscr{E}^i}{\partial r}[0,\xi](x)-\dfrac{\partial\mathscr{E}^s}{\partial r}[0,\xi](x)\right)\\&+\nn\times\left(\xi(x)\cdot\nabla\left(\EE^{i}-\EE^s-\EE^{inc}\right)\right).\end{array}$$
 We recall that $\dfrac{\partial\tau_{r}\nn_{r}}{\partial r}[0,\xi](x)=-\left[\nabla_{\Gamma}\xi\right]\nn$ and we use $$\nabla u=\nabla_{\Gamma}u+\left(\frac{\partial u}{\partial \nn}\right)\nn.$$
We obtain :
$$\begin{array}{rl}\nn(x)\times\left(\dfrac{\partial\mathscr{E}^i}{\partial r}[0,\xi](x)-\dfrac{\partial\mathscr{E}^s}{\partial r}[0,\xi](x)\right)=&\left[\nabla_{\Gamma}\xi\right]\nn\times\left(\EE^{i}(x)-\EE^s(x)-\EE^{inc}(x)\right)\\&-\nn\times\left(\xi(x)\cdot\nabla_{\Gamma}\left(\EE^{i}(x)-\EE^s(x)-\EE^{inc}(x)\right)\right)\\&-(\xi\cdot\nn)\nn\times\dfrac{\partial}{\partial\nn}\left(\EE^{i}(x)-\EE^s(x)-\EE^{inc}(x)\right).\end{array}$$

 Since the tangential component of $\EE^{i}-\EE^s-\EE^{inc}$ vanish we have:
$$\left(\xi(x)\cdot\nabla_{\Gamma}\left(\EE^{i}(x)-\EE^s(x)-\EE^{inc}(x)\right)\right)=\left(\left[\transposee{\nabla_{\Gamma}\nn}\right]\xi\right)\left(\nn\cdot\left(\EE^{i}(x)-\EE^s(x)-\EE^{inc}(x)\right)\right)$$
and $$\left[\nabla_{\Gamma}\xi\right]\nn\times\left(\EE^{i}(x)-\EE^s(x)-\EE^{inc}(x)\right)=\left(\left[\nabla_{\Gamma}\xi\right]\nn\right)\times\nn\left(\EE^{i}(x)-\EE^s(x)-\EE^{inc}(x)\right)\cdot\nn.$$
 For regular surface  we have $\nabla_{\Gamma}\nn=\transposee{\nabla_{\Gamma}\nn}$ and   $$\left(\left[\nabla_{\Gamma}\xi\right]\nn\right)\times\nn-\nn\times\left(\left[\transposee{\nabla_{\Gamma}\nn}\right]\xi\right)=\Rot_{\Gamma}\left(\xi\cdot\nn\right).$$
 We deduce the first boundary conditions. The second boundary condition corresponds to the same computation with the magnetic fields.
\end{proof}
Using the commutators \eqref{commutators} one can verify that the solution of this problem is in $\HH_{loc}(\Rot,\Omega)$ since the trace
$$\uu\in\HH^1(\Rot,\Omega)\mapsto\nn\times\frac{\partial}{\partial \nn}\uu\in\TT\HH^{-\frac{1}{2}}(\Div_{\Gamma},\Gamma)$$
is linear and continuous.

%%%%%%%%%%%%%%%%%%%%%%%%%%%%%%%%%%%%%%%%%%%
\section*{Conclusion}

In this paper we have presented a complete  shape differentiability analysis of the solution to the  dielectric scattering problem using the boundary integral equation approach. These results can be extended to many others electromagnetic boundary value problems. Thanks to the numerous computations of G\^ateaux derivatives we obtain  two  alternatives to compute the first shape derivative of the solution : either we  derive the integral representation or we solve the new boundary value problem associated to the shape derivatives with boundary integral equation method. 
Whereas this last alternative  needs boundaries of class $\mathscr{C}^2$ at least since it appears any derivatives of the normal vector, many results in this paper are still available for Lipschitz domains as for example the computations of the G\^ateaux derivatives of  all the surface differential operators of order 1 with deformations of class $\mathscr{C}^1$ only and other functionals viewed in section \ref{GDiffPH}. One can find in the litterature, the theory of pseudo-differential operators on Lipschitz domain \cite{Taylor2}, it remains to find the optimal regularity of the deformations in order that  this integral operators are still G\^ateaux differentiable. According to the Helmholtz decomposition we have on Lipschitz domain:
 
 $$ \TT\HH^{-\frac{1}{2}}(\Div_{\Gamma},\Gamma)= \nabla_{\Gamma}\left(\mathcal{H}(\Gamma) \right)\bigoplus {\Rot}_{\Gamma}\left(H^{\frac{1}{2}}(\Gamma)\slash\R\right).$$
 where
 
$$\mathcal{H}(\Gamma)=\{u\in H^1(\Gamma)\backslash\R;\;\Delta_{\Gamma}\in H^{-\frac{1}{2}}_{*}(\Gamma)\}.$$
If we want to extend the result to Lipschitz domain, we have to  construct another invertible operator between $\mathcal{H}(\Gamma_{r})$ and $\mathcal{H}(\Gamma)$.

 \nocite{*}
 \small
 \bibliographystyle{siam}

\bibliography{refshapederivative}
   \end{document}